\documentclass[12pt,reqno]{amsart}
\pdfoutput=1 
\usepackage{amsfonts,latexsym,amssymb,dsfont,amsmath}
\usepackage{pdfpages} 
\usepackage{amsthm}
\usepackage{enumitem}
\usepackage[all]{xy}
\usepackage{todonotes}
\usepackage{a4wide}
\usepackage[bookmarks=false]{hyperref}
\usepackage{tikz}

\hypersetup{
    colorlinks,
    linkcolor={red!50!black},
    citecolor={blue!50!black},
    urlcolor={blue!80!black}
}

\newcommand{\R}{{\mathbb R}}
\def\C{{\mathbb C}}

\newcommand{\N}{{\mathbb N}}
\newcommand{\Z}{{\mathbb Z}}
\newcommand{\I}{\mathrm{i}}

\newcommand{\rmi}{\mathrm{i}}

\renewcommand{\Re}{\operatorname{Re}}

\newcommand{\id}{{\ensuremath{\mathds{1}}}}

{}

\makeatletter
\newcommand{\leqnomode}{\tagsleft@true}
\newcommand{\reqnomode}{\tagsleft@false}
\makeatother

 \newtheorem{theorem}{Theorem}[section]

 \theoremstyle{definition}
 
 \newtheorem{exa}[theorem]{Example}

 \newtheorem{exercise}[theorem]{Exercise}

\allowdisplaybreaks
\title[MPS and Geometry]{Computation of Eigenvalues, Spectral Zeta Functions and Zeta-Determinants on Hyperbolic Surfaces}

\author[A. Strohmaier]{Alexander Strohmaier}
\address{Department of Mathematical Sciences,  Loughborough University,  Loughborough, Leicestershire, LE11 3TU, UK} 
\email{a.strohmaier@lboro.ac.uk}

\date{\today}
\keywords{}

\begin{document}

\begin{abstract}
 These are lecture notes from a series of three lectures given at the summer school ``Geometric and Computational Spectral Theory"  in Montreal in June 2015. The aim of the lecture was to explain the mathematical theory behind computations of eigenvalues and spectral determinants in geometrically non-trivial contexts.
  \end{abstract}
\maketitle


\setcounter{secnumdepth}{1}
\setcounter{tocdepth}{1}
\tableofcontents

\section{The Method of Particular Solutions}

The method of particular solutions is a method to find eigenvalues for domains with Dirichlet boundary conditions. It goes back to an idea
by Fox-Henrici-Moler from 1967 (\cite{fox1967approximations}) and was revived by Betcke and Trefethen \cite{betcke2005reviving}
essentially by modifying it to make it numerically stable. 

\subsection{A high accuracy eigenvalue solver in one dimension}

In order to illustrate the method, let us look at it in the simple case of a differential operator on an interval.
Let $[-L,L] \subset \R$ be a compact interval. As usual,  let $-\Delta = -\frac{\partial^2}{\partial x^2}$ be the Laplace operator and assume that $V \in C^\infty([-L,L])$
is a potential. Then the operator $-\Delta +V$ subject to Dirichlet boundary conditions has discrete spectrum. This means there exists a discrete set of values 
$(\lambda_i)_{i \in \mathbb{N}}$ such that the equation
 $$
 (-\Delta +V - \lambda) u =0, \quad  u(-L)=0, \quad u(L)=0.
$$
admits a non-trivial solution $u = \phi_i$.
The eigenvalues $(\lambda_i)$ can be computed as follows.\\

\noindent
 {\bf Step 1.} Solve the initial value problem.\\
For each $\lambda \in \C$ we can solve the initial value problem
$$
 (-\Delta +V - \lambda)u_\lambda =0, \quad  u_\lambda(-L)=0, \quad \frac{d}{dx }u_\lambda(-L)=1.
$$
This can be done either analytically or numerically depending on the type of differential equation.
Then $u_\lambda(+L)$ as a function of $\lambda$ is entire in $\lambda$. The function does not vanish identically
as for example can be shown using integration by parts at $\lambda=\rmi$. The eigenvalues are precisely the zeros of this function.
This provides a direct proof that the eigenvalues form a discrete set.

\noindent
 {\bf Step 2.} Find the zeros of the function $\lambda \mapsto u_\lambda(+L)$ for example using the secant method or Newton's method. This will converge rather fast because
 the function is analytic.\\
 
This algorithm is implemented in the following Mathematica script in the case $$V(x)=5(1-x^2)$$ on the interval $[-1,1]$.

\includepdf[pages={1,2,3,4}]{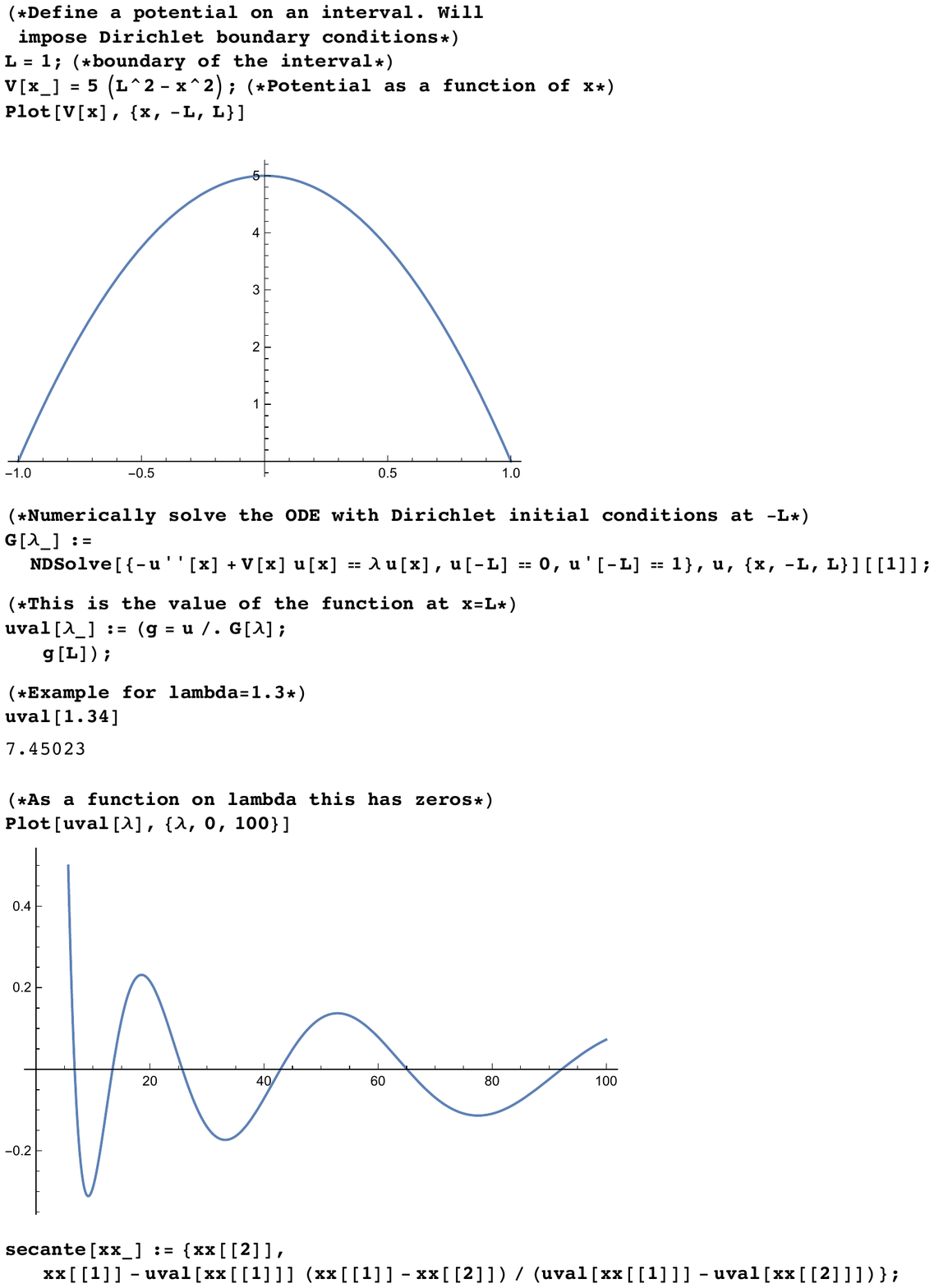}
 
\subsection{Dirichlet eigenvalues for domains in $\R^n$}

The following is a classical result by Fox-Henrici-Moler from 1967 (\cite{fox1967approximations}). Suppose that $\Omega \subset \R^n$ is an open bounded domain in $\R^n$.
Then, the Laplace operator $-\Delta$ with Dirichlet boundary conditions can be defined as the self-adjoint operator obtained from the quadratic form
$$
 q(f,f) = \langle \nabla f, \nabla f \rangle_{L^2(\Omega)}
$$
with form domain $H^1_0(\Omega)$. Since the space $H^1_0(\Omega)$, by Rellich's theorem, is compactly embedded in $L^2(\Omega)$ the spectrum of this operator
is purely discrete and has $\infty$ as its only possible accumulation point. Hence, there exists an orthonormal basis
$(u_j)_{j \in \N}$ in $L^2(\Omega)$ consisting of eigenfunctions with eigenvalues $\lambda_j$, which we assume to be ordered, i.e.
\begin{gather}\nonumber
-\Delta u_j = \lambda_j u_j,\\ 
\| u_j \|_{L^2(\Omega)}=1,\\ \nonumber
u_j |_{\partial \Omega}=0,\\
 0 < \lambda_1 \leq \lambda_2 \leq \cdots \nonumber
\end{gather}
Suppose that $u \in C^\infty(\overline \Omega)$ is a smooth function on the closure $\overline \Omega$ of $\Omega$ satisfying
$$
 -\Delta u = \lambda u,
$$
and assume that
\begin{gather}
 \| u \|_{L^2(\Omega)}=1,\nonumber\\
 \epsilon = \sqrt{|\Omega|}\cdot \| u|_{\partial \Omega } \|_{\infty} <1.
\end{gather}

Then the theorem of Fox-Henrici-Moler states that there exists an eigenvalue $\lambda_j$ of the Dirichlet Laplace operator $-\Delta_D$ such that
\begin{gather}
 \frac{| \lambda - \lambda_j|}{\lambda} \leq \frac{\sqrt{2} \epsilon + \epsilon^2}{1-\epsilon^2}.
\end{gather}

This estimate can be used to obtain eigenvalue inclusions as follows. Choose a suitable set of functions $(\phi_j)_{j=1,\ldots,N}$ satisfying
$$
  -\Delta \phi_j  = \lambda \phi_j.
$$
Such functions could  for example be chosen to be plane waves $\phi_j = \exp(\I \mathbf{k}_j \cdot x)$, where $\mathbf{k}_j \in \R^n$ are vectors such that $\| \mathbf{k}_j \| = \lambda$.
Then one tries to find a linear combination $u=\sum_{j=1}^N v_j \phi_j, \; v_j \in \mathbb{R}$ such that $\| u |_{\partial \Omega}\|_\infty$ is very small. If one approximates the boundary by a finite set of points this reduces to 
a linear algebra problem.
This strategy was quite successful to find low lying eigenvalues for domains in $\R^2$, but was thought to be unstable for higher eigenvalues and for greater precision
when more functions were used. The reason for this unstable behavior is that with too many functions being used, i.e. $N$ being very large, there might be more
linear combinations of the functions $\phi_j$ whose $L^2$-norm is rather small, despite the fact that the $\ell^2$-norm of the coefficient vector $a_j$ is not small.

Betcke and Trefethen (\cite{betcke2005reviving}) managed to stabilize the method of particular solutions by preventing the function $u$ from becoming small in the interior. A simple way to implement
a stable method of particular solutions is as follows. 

Let $(\phi_k)_{k=1,\ldots,N}$ be functions as before.
Let $(x_j)_{j=1,\ldots,M}$ be a family of points on the boundary $\partial \Omega$, and let $(y_j)_{j=1,\ldots,Q}$ be a sufficiently large family of internal points in $\Omega$,
say randomly distributed.

We are looking for a linear combination $u = \sum_{k=1}^N v_k \phi_k$ that is small at the boundary, but that does not vanish in the interior of $\Omega$. Thus, roughly, we are
seeking to minimize $\sum_{j=1}^M | u(x_j) |^2$ whilst keeping $\sum_{j=1}^Q | u(y_j) |^2$ constant. Using the matrices
\begin{gather*}
 A=(a_{ij}), \quad a_{ij}= \phi_j(x_i),\\
 B=(b_{ij}), \quad b_{ij}= \phi_j(y_i),
\end{gather*}
we are thus looking for a vector $v =(v_1,\ldots,v_N) \in \C^N$ such that the quotient $\frac{\| A v \|}{\| B v \|}$ is minimal. Minimizing this quotient is the same as finding the smallest
generalized singular vector of the pair $(A,B)$. The minimal quotient is the smallest singular value of the pair $(A,B)$. 
This value can then be plotted as a function of $\lambda$.\\

The following simple Mathematica code implements this for in the interior of an ellipse.
This is done for the interior 
$$
 \Omega = \{ (x,y) \in \R^2 \mid \frac{1}{4} x^2 + y^2 < 1 \}.
$$
The code illustrates that the first Dirichlet eigenvalues can be computed with a remarkable precision.

\includepdf[pages={1,2}]{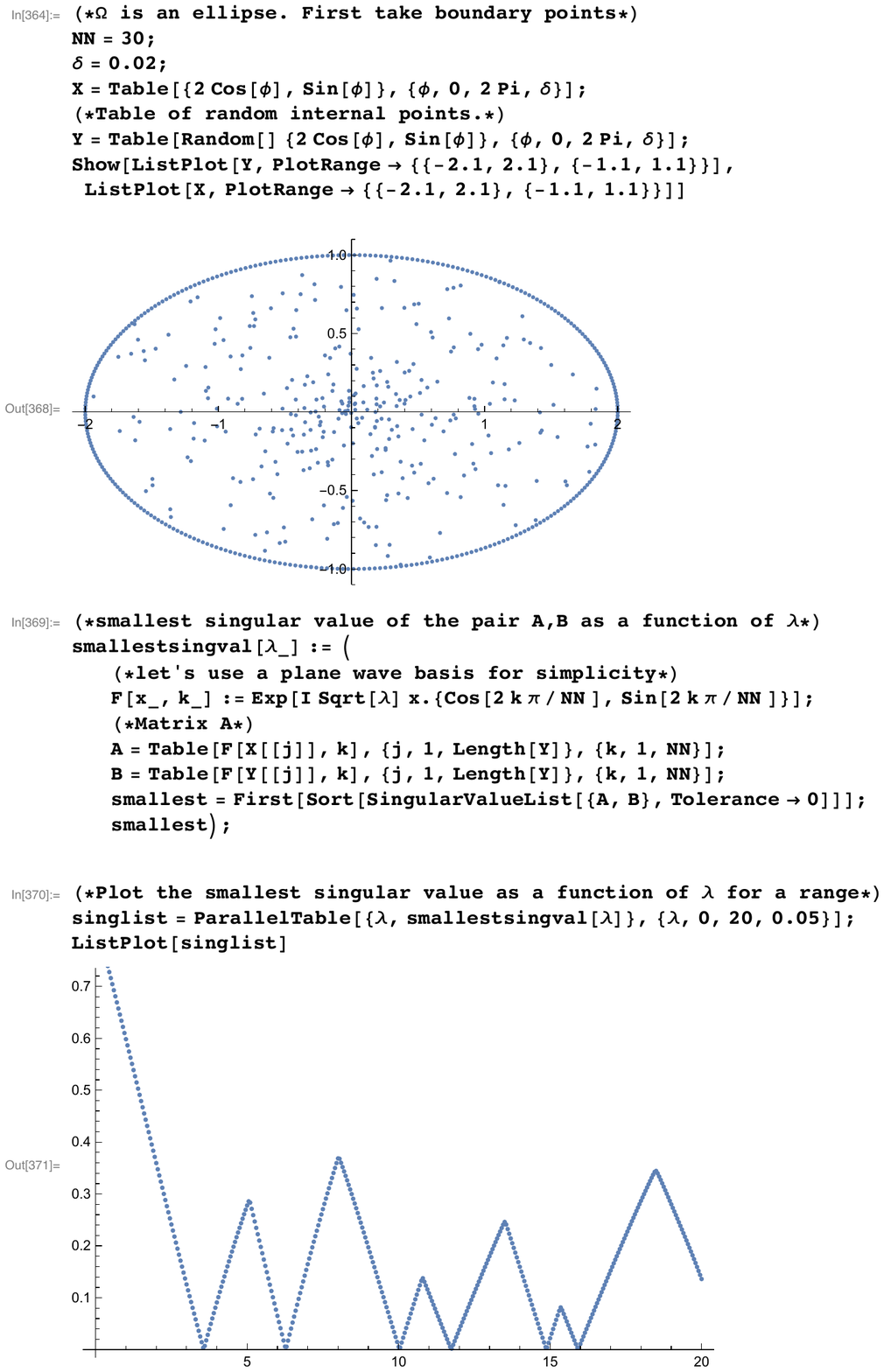}

Once the numerical part is successful and we have a singular vector for the smallest singular value, we are left with two analytical challenges to establish an eigenvalue
inclusion in an interval $[\lambda-\epsilon,\lambda+\epsilon]$:

\begin{enumerate}
 \item Prove that the function $u$ is small on the boundary, i.e. estimate $ \| u |_{\partial \Omega}\|_\infty$.
 \item Prove that the $L^2$-norm of the function $u$ is not too small, i.e. estimate $\| u \|_{L^2(\Omega)}$.
\end{enumerate}

The first point is easy to deal with, for example by Taylor expanding the function $u$ at the boundary (in case the boundary is smooth) and using Taylor's remainder estimate.
The second point is more tricky. Since however even any bad bound from below will do the job, numerical integration with a remainder term can be used to check directly
that the $L^2$-norm is not very small. 

Once a list of eigenvalues is established there is another analytical challenge.
\begin{enumerate}
 \item[(3)] Prove that the method does not miss any eigenvalue if the step-size is chosen small enough.
\end{enumerate}
This point is the most difficult one. It requires a proof that the set of functions is sufficiently large in a quantified sense. It is often easier to first compute a list of eigenvalues and then check afterwards, using other methods, that this list is complete.

The method of particular solutions for domains has been further improved beyond what is presented here (see for example \cite{barnett2011boundary} and references therein) and a software package 
MPS-pack (\cite{barnettmpspack}) exists that makes it possible to compute eigenvalues with very high accuracy for domains in $\R^2$.

\section{The Method of Particular Solutions in a Geometric Context}

Instead of the Dirichlet problem for a domain, we will now consider the problem of finding the spectral resolution of the Laplace operator on a closed Riemannian manifold $M$
with metric $g$ and dimension $n$. Then the metric Laplace operator $-\Delta: C^\infty(M) \to C^\infty(M)$ is given in local coordinates by
\begin{gather}
 -\Delta = - \sum_{i,k=1}^n \frac{1}{\sqrt{|g|}} \frac{\partial}{\partial x^i} \sqrt{|g|} g^{ik}  \frac{\partial}{\partial x^k}.
\end{gather}
The space $C^\infty(M)$ is equipped with the metric inner product
$$
 \langle f_1, f_2 \rangle = \int_M f_1(x) \overline{f_2(x)} \sqrt{|g|} dx.
$$
The completion of $C^\infty(M)$ is the space $L^2(M)$.
The Laplace operator is essentially self-adjoint as an unbounded operator in $L^2(M)$ and the domain of the closure is equal to the second Sobolev space $H^2(M)$.
By Rellich's theorem this space is compactly embedded in $L^2(M)$ and therefore the Laplace operator has compact resolvent, i.e. its spectrum is purely discrete with $\infty$
as the only possible accumulation point. Moreover, $-\Delta$ is a non-negative operator, and the zero eigenspace consists of locally constant functions.
Because of elliptic regularity the eigenfunctions are smooth on $M$. Summarizing, we therefore know that there exists an orthonormal basis $(u_j)$ in $L^2(M)$
such that
\begin{gather}
-\Delta u_j = \lambda_j u_j,\nonumber\\
u_j \in C^\infty(M),\\
 0 \leq \lambda_1 \leq \lambda_2 \leq \cdots \nonumber
\end{gather}

We will be applying the idea of the method of particular solutions to manifolds (see \cite{strohmaier2013algorithm}). We start by describing this in a very general setting.
Suppose that $M$ is a compact Riemannian manifold and suppose this manifold is glued from a finite number of closed subsets $M_j$ along their boundaries so that
$$
 M = \cup_{j=1}^q M_j.
$$
We assume here that $M_j$ are manifolds that have a piecewise smooth Lipschitz boundary.

\begin{exa}
 The $n$-torus $T^n$ can be obtained from the cube $[0,1]^n$ by identifying opposite boundary components. In this case we have only one component $M_1$ and its boundary
 $\partial M_1$. 
\end{exa}

\begin{exa}
 A surface of genus $2$ can be glued from two pair of pants, or alternatively, from $4$ hexagons. This will be discussed in detail in Section \ref{hypsuf}.
\end{exa}

If $f \in C^\infty(M)$ is a function on $M$ then we can of course restrict this function to each of the components $M_j$ and we thus obtain a natural map
\begin{gather}
 R: C^\infty(M) \to C^\infty(\sqcup_j M_j).
\end{gather}
Since the interior of $M_j$ is naturally a subset in $M$, and its boundary has zero measure,
we can also understand functions in $C^1(\sqcup_j M_j)$ as (equivalence classes of) functions on $M$ that have jump type discontinuities along the boundaries
of $M_j$. In this way we obtain a map
\begin{gather}
  E: C^\infty(\sqcup_j M_j) \to L^\infty(M).
\end{gather}
By construction, we have $E \circ R = \id$.
Given a function in $C^\infty(\sqcup_j M_j)$, we can also measure its jump behavior as follows.
After gluing the boundaries $\sqcup_j \partial M_j$ form a piecewise smooth Lipschitz hypersurface $\Sigma$ in $M$. Suppose $x$ is a point in $\Sigma$.
Then $x$ arises from gluing points in  $\sqcup_j \partial M_j$. 
We will assume that there are precisely two such points $x_+ \in \partial M_{j_1}$ and $x_- \in \partial M_{j_2}$ that form the point $x$ after gluing.
We will also assume that the normal outward derivatives $\partial_{n(x_+)}$ and $\partial_{n(x_-)}$ are well defined at these points. 
These two assumption are  satisfied on a set of full measure in $\Sigma$. Note that there is 
freedom in the choice of $x_+$ and $x_-$ for a given $x$. We assume here that such a choice has been made and that this choice is piecewise continuous.
Given $f \in C^\infty(\sqcup_j M_j)$
we define
\begin{gather}
 Df(x) =  f(x_+) - f(x_-),\nonumber\\
 D_n f(x) = \partial_{n(x_+)} f +  \partial_{n(x_-)} f.
\end{gather}
These functions are then functions in $L^\infty(\Sigma)$. $Df$ measures the extent to which $f$ fails to be continuous and $D_n f$ measures the extent to which
$f$ fails to be differentiable.

The significance of the functions $Df$ and $D_n f$ is in the fact that they naturally appear in Green's identity as follows. Suppose that $(f_j)$ is a collection
of smooth functions on $M_j$ and $f$ is the assembled function $f=E (f_j)$. Then, by Green's formula, for any test function $g \in C^\infty_0(M)$ we have
\begin{gather}
 \int_{M} f(x) (\Delta g)(x) dx =
  \sum_{j} \left(  \int_{M_j} (\Delta f_j) (x)  g(x) dx  \right)\nonumber \\ + \sum_j \left(-\int_{\partial M_j} (\partial_n f)(x) g(x) dx +  \int_{\partial M_j} f(x) (\partial_n g)(x) dx \right).
\end{gather}
The last two terms can be re-written as
\begin{gather}
 \sum_j \left(-\int_{\partial M_j} (\partial_n f)(x) g(x) dx +  \int_{\partial M_j} f(x) (\partial_n g)(x) dx \right) \nonumber \\=
 -\int_{\Sigma} (D_n f)(x) g(x) dx +  \int_{\Sigma} (Df)(x) (\partial_n g)(x) dx
\end{gather}
if the normal vector field $\partial_n$ at the point $x$ is chosen to be $\partial_{n(x_+)}$.
In other words, in the sense of distributional derivatives  $-\Delta f$ is the distribution
\begin{gather}
  E(-\Delta f_j) + (D_n f) \otimes \delta_{ \Sigma} + (D f) \otimes \delta'_{ \Sigma}.
\end{gather}
Here the distributions $\delta_{\Sigma}$ and $\delta'_{\Sigma}$ are the Dirac delta masses and the corresponding normal derivative
along the hypersurface $\Sigma$. The tensor product here is understood in the sense that pairing with test functions is defined as follows
\begin{gather}
(h \otimes \delta_{\Sigma})(g) := \int_{\Sigma} h(x) g(x) dx
\end{gather}
and
\begin{gather}
 (h \otimes \delta'_{\Sigma})(g) :=  - \int_{\Sigma} h(x) (\partial_n g)(x) dx.
\end{gather}

In particular, if the functions $f_j$ satisfy the eigenvalue equation $(\Delta + \lambda) f_j =0$ on each component $M_j$ then 
we have in the sense of distributions
\begin{gather}
 (-\Delta - \lambda) f = (D_n f) \otimes \delta_{ \Sigma} + (D f) \otimes \delta'_{\Sigma}.
\end{gather}

Since $\Sigma$ was assumed to be piecewise smooth and Lipschitz, the Sobolev spaces $H^s(\Sigma)$ are well defined for any $s \in \R$. 
\newpage

\begin{theorem}
 There exists a constant $C>0$ which can be obtained explicitly for a given Riemannian manifold $M$ and decomposition $(M_j)$ once the Sobolev norms are defined in local coordinates, such that the following statement holds.
 Suppose that $(\phi_j)$ is a collection of smooth functions on $M_j$, and denote by $\phi$ the corresponding function $E(\phi_j)$ on $M$.
 Suppose furthermore that
 \begin{enumerate}
  \item $\| \phi\|_{L^2(M)}=1$,
  \item $-\Delta \phi - \lambda \phi = \chi$ on $M \backslash \Sigma$,
  \item $\| \chi \|_{L^2(M)}= \eta$,
  \item $C  \left( \| D \phi \|^2_{H^{-\frac{1}{2}}(\Sigma)} + \| D_n \phi \|^2_{H^{-\frac{3}{2}}(\Sigma)} \right)^{\frac{1}{2}} = \epsilon <1$.
 \end{enumerate}
 Then there exists an eigenvalue $\lambda_j$ of $-\Delta$ in the interval
 $$
  [\lambda - \frac{(1+\lambda) \epsilon + \eta}{1-\epsilon},\lambda+\frac{(1+\lambda) \epsilon + \eta}{1-\epsilon}].
 $$
\end{theorem}

\begin{proof}
 By the Sobolev restriction theorems the distributions $(D_n f) \otimes \delta_{\partial \Sigma}$ as well as $(D f) \otimes \delta'_{\partial \Sigma}$
 are in $H^{-2}(M)$ and we have
 \begin{gather*}
  \| D\phi \otimes \delta'_{ \Sigma}\|_{H^{-2}(M)} \leq C_1 \|  D \phi \|_{H^{-1/2}(\Sigma)},\\
  \| D_n\phi \otimes \delta_{ \Sigma}\|_{H^{-2}(M)} \leq C_2 \|  D_n \phi \|_{H^{-3/2}(\Sigma)}.
 \end{gather*}
Loosely speaking this follows since restriction to a co-dimension one Lipschitz hypersurface is continuous as a map from $H^s$ to $H^{s-\frac{1}{2}}$ for $s> \frac{1}{2}$ and the
corresponding dual statement. These estimates can also be obtained in local coordinates using the Fourier transform. The constants $C_1$ and $C_2$ can therefore be estimated
once local charts are fixed.

Let us define the distribution $g:= (-\Delta +1)^{-1}\left(  (D_n f) \otimes \delta_{ \Sigma} + (D f) \otimes \delta'_{\Sigma} \right)$. Then, by elliptic regularity,
$g \in L^2(M)$ and
$$
\| g \|_{L^2(M)} = \epsilon \leq C \left( \| D \phi \|^2_{H^{-\frac{1}{2}}(\Sigma)} + \| D_n \phi \|^2_{H^{-\frac{3}{2}}(\Sigma)} \right)^{\frac{1}{2}}.
$$
One checks by direct computation that
$$
 (-\Delta - \lambda) (\phi -g)  = \chi + (1 + \lambda) g.
$$
Using 
\begin{gather*}
 \| \chi +(1 + \lambda) g \|_{L^2(M)} \leq \eta + |1+\lambda| \| g\|_{L^2(M)},\\
 \| \phi-g \| \geq 1 - \| g \|_{L^2(M)},
\end{gather*}
one obtains
$$
 \| (-\Delta - \lambda)^{-1}\|_{L^2(M)} \geq \frac{1-\| g\|_{L^2(M)}}{\eta +|1+\lambda| \|g\|_{L^2(M)}}.
$$
This implies the statement as the resolvent norm is bounded by the distance to the spectrum.
\end{proof}

Of course, $ \| g \|^2_{H^{s}(\Sigma)} \leq \| g \|^2_{L^2(\Sigma)}$ for any $s \leq 0$ so, one also obtains a bound in terms
of 
$\left( \| D \phi \|^2_{L^2(\Sigma)} + \| D_n \phi \|^2_{L^2(\Sigma)} \right)^{\frac{1}{2}}$, although this bound does not take into account
the different microlocal properties of $D_n \phi$ and $D \phi$, i.e. their behaviour for large frequencies.

\section{Hyperbolic Surfaces and Teichm\"uller Space} \label{hypsuf}

The following section is a brief description of the construction and theory of hyperbolic surfaces. In the same way as the sphere $S^2$ admits a round metric
and the torus $T^2$ admits a two dimensional family of flat metrics, a two dimensional compact manifold $M$ of genus $\mathrm{g} \geq 2$ admits a family of metrics of constant negative curvature $-1$.
By the theorem of Gauss-Bonnet all these metrics yield the same volume
$$
 \mathrm{Vol}(M) = 4 \pi (\mathrm{g} -1 ).
$$
For an introduction into hyperbolic surfaces and their spectral theory, we would like to refer to the reader to the excellent monograph \cite{buser2010geometry}.
We start by describing some two dimensional spaces of constant curvature $-1$.

\begin{itemize}

\item{\bf The upper half space} 

The hyperbolic upper half space  $\mathbb{H}$ is defined as $\mathbb{H}:=\{ x+iy \in \C \mid y>0\}$ with metric
$$
 g = y^{-2} (dx^2 + dy^2).
$$
The Laplace operator with respect to this metric is then given by
$$
 - \Delta = -y^2 \left( \frac{\partial^2}{\partial x^2} + \frac{\partial^2}{\partial y^2}\right).
$$
The geodesics in  this space are circles that are perpendicular to the real line. The group of isometries of the space is the group
$PSL(2,\R)$. The action of $PSL(2,\R)$ derives from the action of $SL(2,\R)$ on $\mathbb{H}$ by fractional linear transformations as follows.
$$
 \left( \begin{matrix} a & b \\ c & d \end{matrix} \right) z = \frac{a z +b}{ c z + d}.
$$
Since  $\left( \begin{matrix} -1 & 0 \\ 0 & -1 \end{matrix} \right)$ acts trivially, this factors to an action of  $PSL(2,\R) =SL(2,\R) / \{-1,1\}$. It is easy to check that this acts as a group of isometries.

\item{\bf The Poincar\'e disc}

The Poincare disc $\mathbb{D}$ is defined as $\mathbb{D}:=\{ x+iy \in \C \mid x^2 + y^2 < 1\}$ with metric
$$
 g = \frac{4}{(1-x^2-y^2)^2} (dx^2 + dy^2).
$$
Geodesics in this model are circles perpendicular to the unit circle and straight lines through the origin.
This space has constant negative curvature $-1$ and is simply connected. It therefore is isometric to the hyperbolic plane. An isometry from $\mathbb{D}$ to $\mathbb{H}$ is for example the Moebius transformation
$$
 z \mapsto \I \frac{1 + z}{1-z}.
$$

\item{\bf Hyperbolic cylinders}

Let $\ell>0$. Then the hyperbolic cylinder can be defined as the quotient $Z_\ell:=\Gamma \backslash \mathbb{H}$ of $\mathbb{H}$ by the group $\Gamma \subset SL(2,\R)$ defined by
$$
 \Gamma= < \left( \begin{matrix}  e^{\ell/2} & 0 \\ 0 & e^{-\ell/2} \end{matrix} \right) > = \left \{ \left( \begin{matrix}  e^{\ell k /2} & 0 \\ 0 & e^{-\ell k/2} \end{matrix} \right) \mid k \in \Z \right \}.
$$
A fundamental domain is depicted in the Figure \ref{cyl-fund}.
\begin{figure}[h] 
\centering
\includegraphics*[width=10cm]{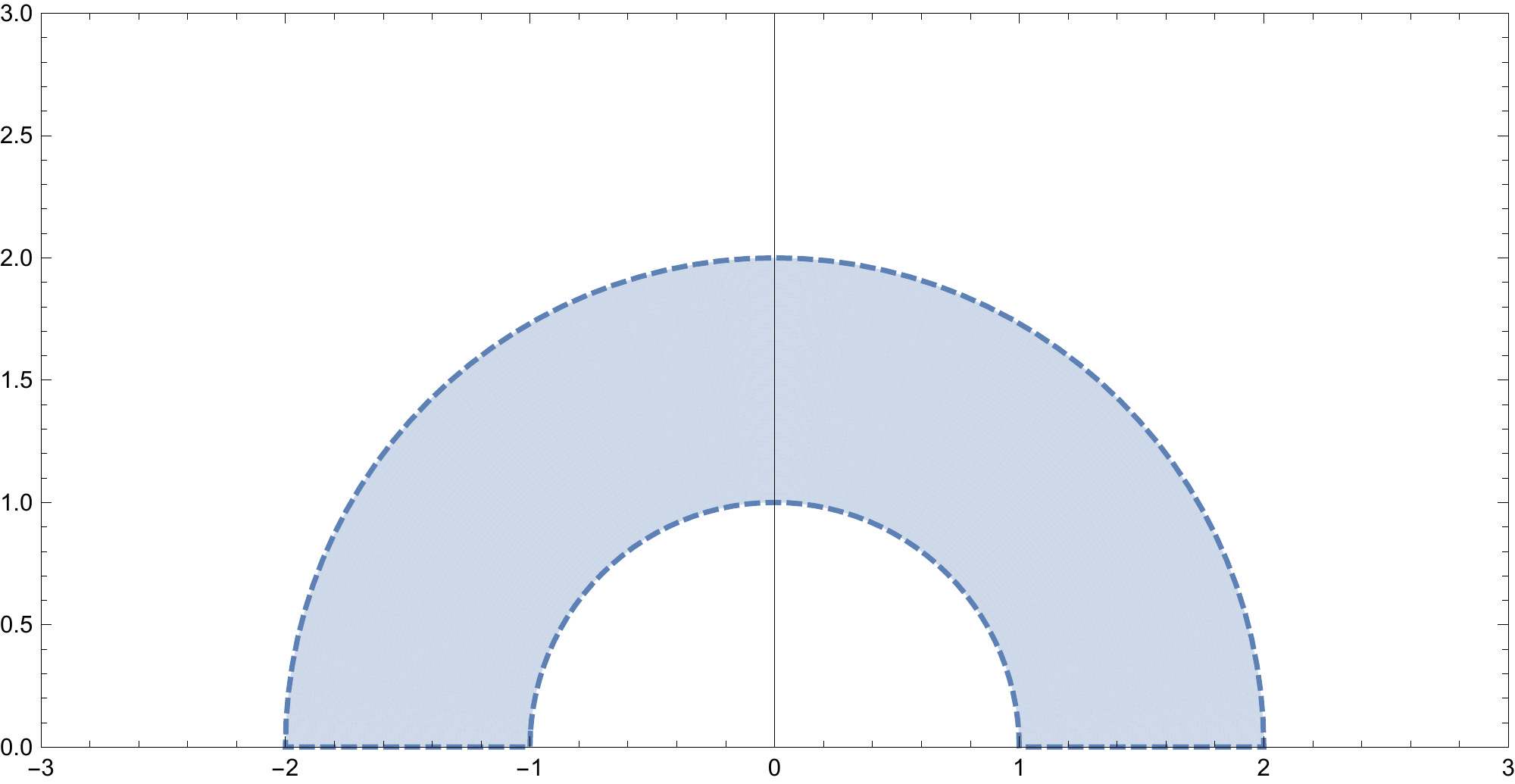}
\caption{Fundamental domain for a hyperbolic cylinder}\label{cyl-fund}
\end{figure}  Using the angle $\varphi = \arctan(x/y)$ and $t = \frac{1}{2}\log(x^2 +y ^2)$ as coordinates the metric becomes
$$
 g = \frac{1}{\cos^2 \varphi}(d\varphi^2 + dt^2).
$$
We can also use Fermi coordinates $(\rho,t)$, where $t$ is as before and $\cosh \rho =  \frac{1}{\cos \varphi}$. The coordinate $\rho$ is the oriented hyperbolic distance from the
$y$-axis in $\mathbb{H}$. On the quotient $Z_\ell$ the $y$-axis projects to a closed geodesic of length $\ell$. This is the unique simple closed geodesic on $Z_{\ell}$.
Using Fermi coordinates we can see that the hyperbolic cylinder $Z_\ell$ is isometric to $\mathbb{R} \times (\R / \ell \Z)$ with metric
$$
 d \rho^2 + \cosh^2 \rho \;dt^2.
$$
The Laplace operator in these coordinates
\begin{gather}
 -\frac{1}{\cosh{\rho}}\frac{\partial}{\partial\rho} \cosh{\rho} \frac{\partial}{\partial\rho} - \frac{1}{ \cosh^2{\rho}}\frac{\partial^2}{\partial t^2}.
\end{gather}
A large set of solutions of the eigenvalue equation $(-\Delta - \lambda) \Phi=0$ can then be obtained by separation of variables.
Namely, if we assume that
$$
 \Phi(\rho,t) = \Phi_k(\rho) \exp(2 \pi i \frac{t}{\ell})
$$
for some $k \in \Z$
then the eigenvalue equation is equivalent to 
\begin{gather} \label{dequ}
 (-\frac{1}{\cosh{\rho}}\frac{d}{d\rho} \cosh{\rho} \frac{d}{d\rho} + \frac{4 \pi^2 k^2}{\ell^2 \cosh^2{\rho}} -\lambda) \Phi_k(\rho)=0
\end{gather}
A fundamental system of (non-normalized) solutions of this equation, consisting of an even and an odd function, can be given explicitly for each $k \in \Z$ in terms of hypergeometric functions
\begin{gather}
 \Phi_k^{even}(\rho)= (\cosh{\rho})^{\frac{2 \pi \I k}{\ell}} \; {}_2 \mathrm{F}_1(\frac{s}{2}+\frac{\pi \I k}{\ell},\frac{1-s}{2}+\frac{\pi \I k}{\ell};\frac{1}{2};-\sinh^2\rho),\\
 \Phi_k^{odd}(\rho)= \sinh{\rho} (\cosh{\rho})^{\frac{2 \pi \I k}{\ell}} \; {}_2 \mathrm{F}_1(\frac{1+s}{2}+\frac{\pi \I k}{\ell},\frac{2-s}{2}+\frac{\pi \I k}{\ell};\frac{3}{2};-\sinh^2\rho),\nonumber
\end{gather}
where $\lambda=s(1-s)$ (see \cite{borthwick2010sharp}, where these functions are analysed). Normalization gives the corresponding solutions to the initial value problems.

\item{\bf Hyperbolic pair of pants}

\begin{figure}[h] 
\centering
\includegraphics*[width=5cm]{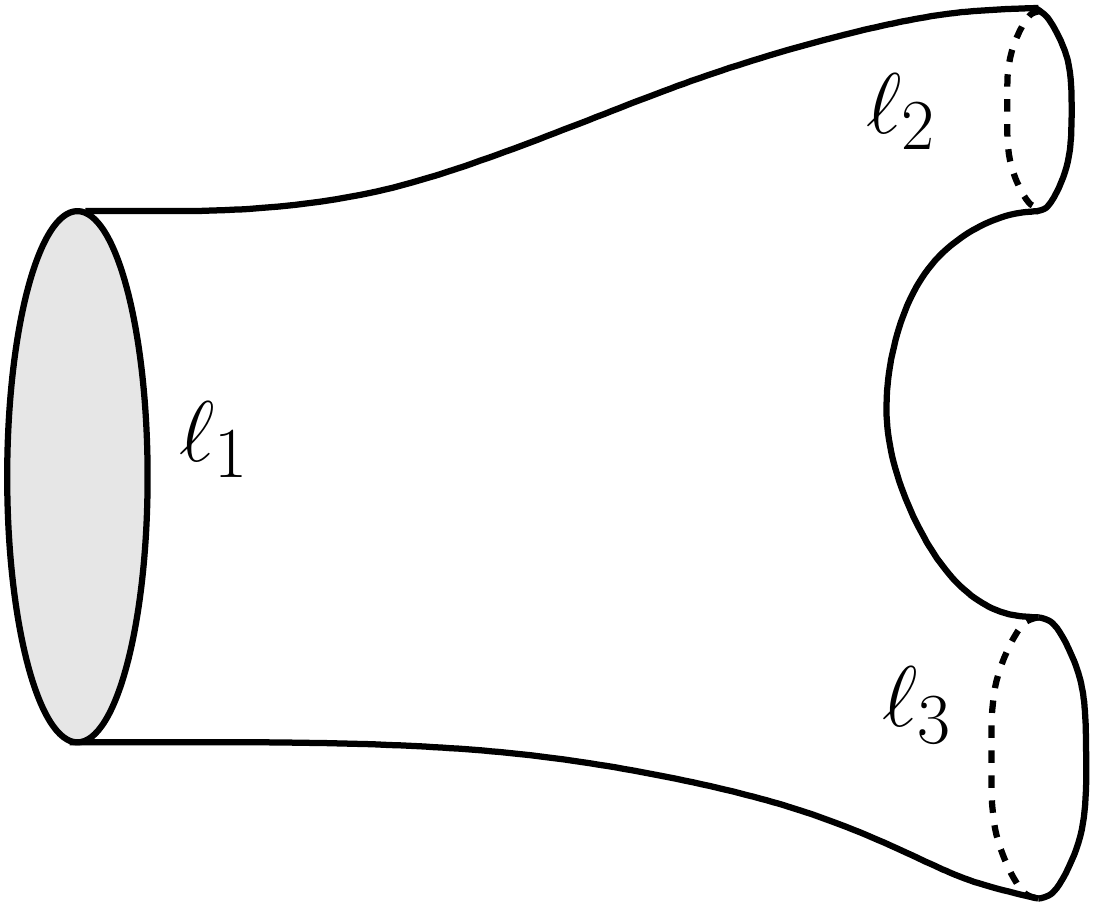}
\caption{$Y$-piece with boundary geodesics}\label{pants}
\end{figure} 

For any given $\ell_1, \ell_2, \ell_3 >0$ one can construct a right angled geodesic hexagon in the hyperbolic plane such that the length of every second side is
$\ell_1/2, \ell_2/2$ and $\ell_3/2$. Two such hexagons can then be glued along the other sides to form a hyperbolic surface with three geodesic boundary components
of lengths $\ell_1, \ell_2, \ell_3$.
A hyperbolic pair of pants can also be glued from a subset of a hyperbolic cylinder as depicted in the figure.
\begin{figure}[h] 
\centering
\includegraphics*[width=8cm]{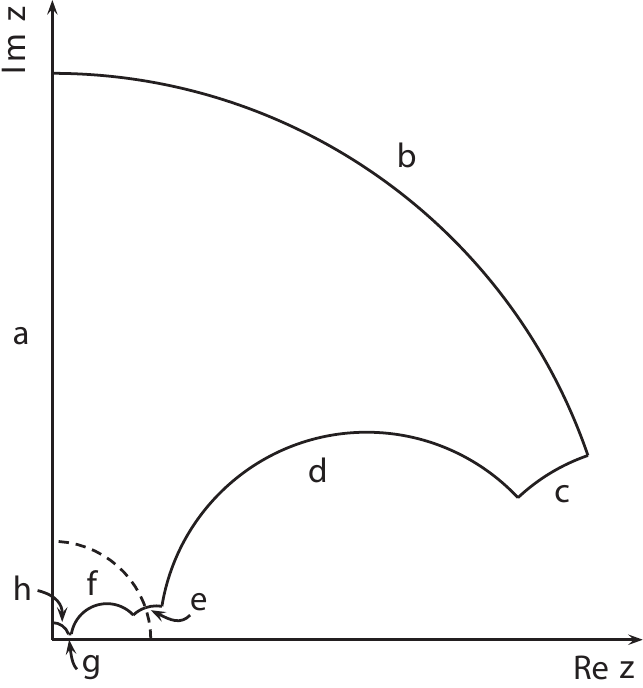}
\caption{Two hyperbolic hexagons together form an octagon which can be glued into a pair of pants}\label{pants2}
\end{figure} 

\item{\bf General surfaces of genus $\mathrm{g}$}

Let $\mathrm{g}>2$ be an integer. Suppose we are given $2\mathrm{g}-2$ pairs of pants, and a three-valent graph together with a map that associates with each vertex a pair of pants,
and with each edge  associated with that vertex a boundary component of that pair of pants. So each edge of the graph will connect two vertices and will therefore correspond
to two different boundary components of that pair of pants. Suppose that these boundary components have the same length. So each edge of the graph will have a length
$\ell_j$ associated to it. There are $3 \mathrm{g}-3$ such edges. We can then glue the hyperbolic pair of pants together along the boundary components using a gluing scheme
that identifies each collar neighborhood of the boundary component with a subset of the corresponding hyperbolic cylinder. Such a gluing is unique up to a twist angle
$\alpha_j \in S^1$. Once such a twist angle is fixed we obtain a surface of genus $\mathrm{g}$ equipped with a hyperbolic metric. It is known that each oriented hyperbolic surface
can be obtained in this way. The parameters $\ell_j$ and $\alpha_j$ then constitute the Fenchel-Nielsen parameters of that construction. For each given three-valent graph
and $6\mathrm{g}-6$ Fenchel-Nielsen parameters there is a hyperbolic surface constructed. Of course, it may happen that different Fenchel-Nielsen parameters yield an isometric
surface. It can be shown that there is a discrete group, the mapping class group, acting on the Teichm\"uller space $\R^{6\mathrm{g}-6}$ such that the quotient
coincides with the set of hyperbolic metrics on a given two dimensional oriented surface.
\begin{figure}[htp] 
\centering
\includegraphics*[width=11cm]{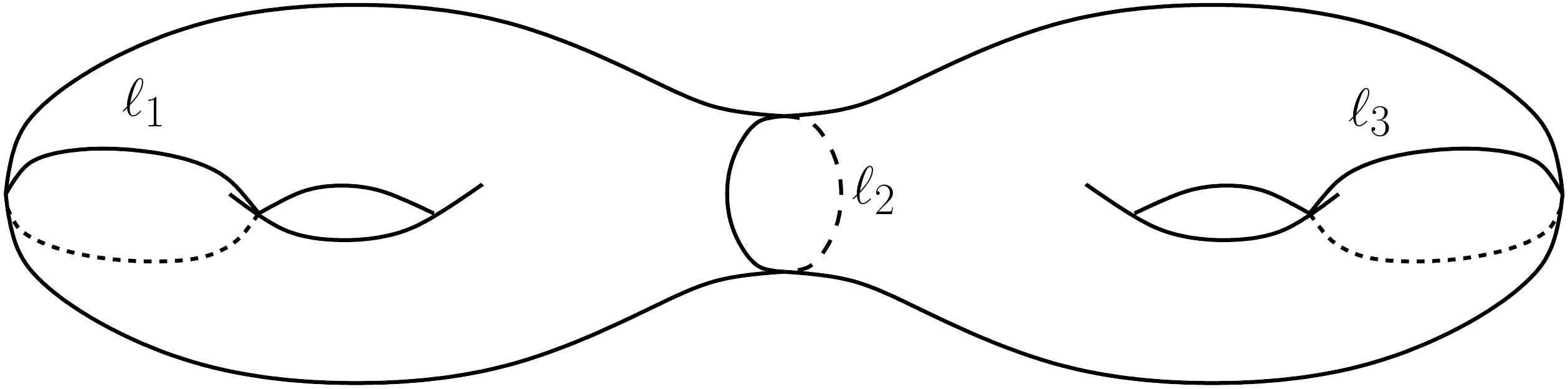}
\caption{Genus two hyperbolic surface glued from two pairs of pants}\label{pants3}
\end{figure} 
\begin{figure}[htp] 
\centering
\includegraphics*[width=11cm]{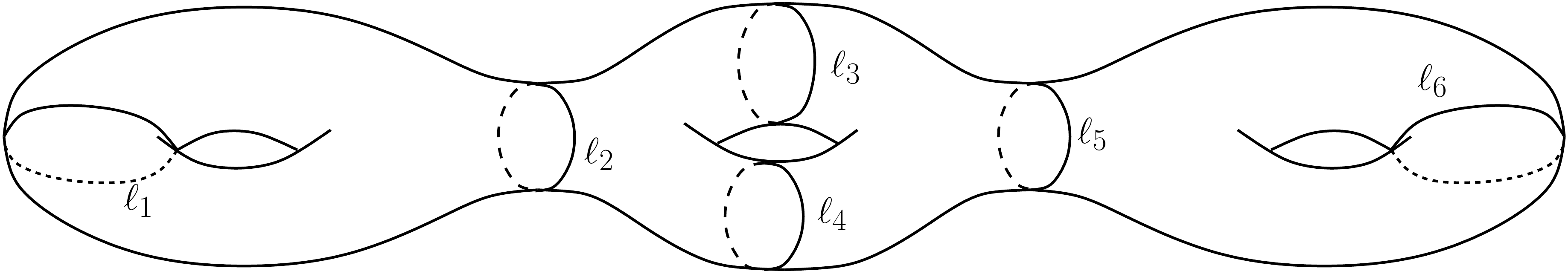}
\caption{Genus three hyperbolic surface glued from four pair of pants}\label{pants4}
\end{figure} 

\end{itemize}

\section{The Method of Particular Solutions for Hyperbolic Surfaces}

In the following, we will describe a very efficient way to implement the method of particular solutions for hyperbolic surfaces.
Each surface can be decomposed into $2 \mathrm{g}-2$ pairs of pants. Each pair of pants can then be cut open along one geodesic connecting two boundary components
to obtain a subset of a hyperbolic cylinder. Our surface $M$ can therefore be glued from $2\mathrm{g}-2$ subsets $M_j$ of hyperbolic cylinders.
$$
 M= \cup_j M_j.
$$
This gives a decomposition of $M$ as discussed before and the hypersurface $\Sigma$ will consist of geodesic segments.
On each piece $M_j$ we have a large set of functions satisfying the eigenvalue equation
$(-\Delta - \lambda) \Phi=0$ by restricting the functions constructed on the hyperbolic cylinder to $M_j$.
If we let $k$ vary between $-N$ and $+N$ we obtain a $2(2N+1)$-dimensional space of functions with a canonical basis.
We can assemble these into a $2(2N+1)(2\mathrm{g}-2)$-dimensional subspace $\mathcal{W}_N^{(\lambda)}$ in $L^\infty(M)$. 
Basis elements in this subspace are indexed by $j \in \{1,\ldots, 2\mathrm{g}-2\}$, by $k \in \{-N, -N+1,\ldots,N-1,N\}$ and by $\{e,o\}$
where the last index distinguishes between even and odd solutions of the ODE. We will assemble all these indices into a larger index
$\alpha$. So we have a set of basis function $\Phi_\alpha^{(\lambda)}$ on $\sqcup_j M_j$ and we would like to apply the estimate MPS in order to find eigenvalues.

A simple strategy is as follows. Discretize the geodesic segments of $\Sigma$ into a finite set of $Q$ points $(x_j)_{j=1,\ldots,Q}$. In order to keep things simple let us avoid corners.
So every point $x_j$ will be contained in the boundary of precisely two components, so there are exactly two points $y_j$ and $\tilde y_j$ in
$\sqcup_j \partial M_j$ that correspond to this point.
A simple strategy of MPS for these surfaces is therefore to form the matrices
\begin{gather*}
  A_\lambda=(a_{j \alpha}), \quad a_{j \alpha}= \Phi_\alpha^{(\lambda)}(y_j),\\
  \tilde A_\lambda=(\tilde a_{j \alpha}), \quad \tilde a_{j \alpha}= \Phi_\alpha^{(\lambda)}(\tilde y_j),\\
  B_\lambda=(b_{j \alpha}), \quad b_{j \alpha}= \partial_n \Phi_\alpha^{(\lambda)}(y_j)\\
  \tilde B_\lambda=(\tilde b_{j \alpha}), \quad \tilde b_{j \alpha}=  \partial_n \Phi_\alpha^{(\lambda)}(\tilde y_j).
\end{gather*}

We assemble $Q_\lambda:= (A _\lambda- \tilde A_\lambda) \oplus (B_\lambda + \tilde B_\lambda)$ and also $R_\lambda:=A_\lambda  \oplus \tilde A_\lambda \oplus B_\lambda \oplus \tilde B_\lambda$.
Then the smallest singular value
$$
 s_\lambda = \inf_{v \not=0} \frac{\| Q_\lambda v\|}{\| R_\lambda v \|}
$$
of the pair $(Q_\lambda,R_\lambda)$ is then a measure of how close we are to an eigenvalue.

For a quantitative statement see \cite{strohmaier2013algorithm} where this method is described and analysed in great detail. 
The idea behind this is easily explained as follows. Suppose that $\lambda$ is an eigenvalue. Then there exists a corresponding eigenfunction
$\phi$. This eigenfunction can be restricted to each piece $M_j$ and can then be expanded in our basis functions. Since the eigenfunction is analytic, the Fourier series with respect to the circle action on the hyperbolic cylinder converges exponentially fast. This means the eigenfunction is approximated exponentially well by the chosen basis functions
$\Phi_\alpha$. Cutting off at a Fourier mode will produce an error in the $C^1$-norm that is exponentially small as $N$ becomes large. 
Since the actual eigenfunction satisfies $D \phi=0$ and $D_n \phi=0$ its approximation by our basis functions $\phi_N$ will satisfy the same equation up to an exponentially small error.
Therefore, if $v$ is the coefficient vector of $\phi_N$ with respect to our basis $\Phi_\alpha$, the norm of $Q_\lambda v$ is very small.
On the other hand, by Green's formula, the boundary data of $\phi$ does not vanish on $\partial M_j$ but merely gives a measure for its $L^2$-norm.
So the norm of $R_\lambda v$ will be comparable to the $L^2$-norm of $\phi$. We conclude that $s_\lambda$ is exponentially small as $N$ gets large if $\lambda$
is an eigenvalue.

Conversely, since $Q_\lambda v$ roughly approximates the $L^2$-norm of $D \phi \oplus D_n \phi$ and  $R_\lambda v$ roughly approximates the $L^2$ norm
of $\phi$, the quotient will not be small if $\lambda$ is not a eigenvalue. 

Hence, if we plot $s_\lambda$ as a function of $\lambda$ we will be able to find the eigenvalues. In a similar way,
multiplicities can be found by looking at higher singular values.

\subsection{The Bolza surface}

In the following, we would like to illustrate this method and some results for the case of the Bolza surface. 
The Bolza surface is the unique oriented hyperbolic surface of genus $2$ with maximal group of orientation preserving isometries of order $48$. 
It can be described in several different ways.

The easiest way uses the Poincare disk model. Define the regular geodesic octagon with corner points
$2^{-\frac{1}{4}} \exp(\frac{\pi i k}{4})$. In order to obtain the Bolza surface,
opposite sides are identified by means of hyperbolic isometries
using the identification scheme as in the figure.
\begin{figure}[h]
\begin{center}
\begin{tikzpicture}[scale=3];
\draw (1,0) arc[x radius = 1, y radius = 1, start angle= 0, end angle= 360];
\draw  [magenta,line width=2pt] (0,0.840896) arc[x radius = 0.4550898605622274, y radius = 0.4550898605622274, start angle= 202.5, end angle= 292.5];
\draw [rotate=45,orange,line width=2pt] (0,0.840896) arc[x radius = 0.4550898605622274, y radius = 0.4550898605622274, start angle= 202.5, end angle= 292.5];
\draw [rotate=90,blue,line width=2pt]  (0,0.840896) arc[x radius = 0.4550898605622274, y radius = 0.4550898605622274, start angle= 202.5, end angle= 292.5];
\draw [rotate=135,green,line width=2pt] (0,0.840896) arc[x radius = 0.4550898605622274, y radius = 0.4550898605622274, start angle= 202.5, end angle= 292.5];
\draw [rotate=180,magenta,line width=2pt] (0,0.840896) arc[x radius = 0.4550898605622274, y radius = 0.4550898605622274, start angle= 202.5, end angle= 292.5];
\draw [rotate=225,orange,line width=2pt]  (0,0.840896) arc[x radius = 0.4550898605622274, y radius = 0.4550898605622274, start angle= 202.5, end angle= 292.5];
\draw [rotate=270,blue,line width=2pt] (0,0.840896) arc[x radius = 0.4550898605622274, y radius = 0.4550898605622274, start angle= 202.5, end angle= 292.5];
\draw [rotate=315,green,line width=2pt] (0,0.840896) arc[x radius = 0.4550898605622274, y radius = 0.4550898605622274, start angle= 202.5, end angle= 292.5];
\end{tikzpicture}
\end{center}
\caption{The Bolza surface obtained from a regular octagon in the hyperbolic plane}
\end{figure}
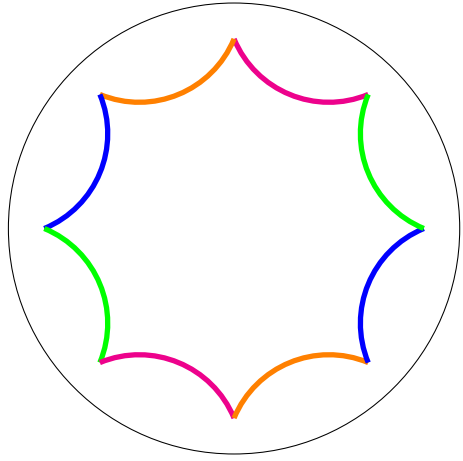

The group of orientation preserving isomtries is $GL(2,\mathbb{Z}_3)$ which is a double cover of $S_4$. The full isometry group 
$GL(2,\mathbb{Z}_3)\rtimes \mathbb{Z}_{2}$ has $13$ isomorphism classes of irreducible representations: four one-dimensional, two two-dimensional, four three-dimensional, and three four-dimensional ones. 
The representation theory of this group and its connection
to boundary conditions on subdomains has been worked out in detail by Joe Cook in his thesis (\cite{cook:thesis}).
It was claimed by Jenni in his PhD thesis
that the first non-zero eigenspace is a three dimensional irreducible representation. The proof seems to rely on some numerical input as well.
Jenni also gives the bound for the first non-zero eigenvalue $3.83<\lambda_1<3.85$.

The Bolza surface was also investigated by Aurich and Steiner in the context of quantum chaos (see for example \cite{aurich1990energy, aurich1989periodic}), where it was referred to as the Hadamard-Gutzwiller model.
A finite element method was applied to the surface and the first non-zero eigenvalue was indeed found to be of multiplicity three and was given by $\lambda_1=3.838$.
Nowadays, it is not difficult to code the Bolza surface in the available finite element frameworks. It can be done rather quickly in the freely available FreeFEM++ (\cite{hecht2012new}). 
Its Fenchel-Nielsen m-w-coordinates can be worked out to be
\begin{gather*}
(\ell_1,t_1;\ell_2,t_2;\ell_3,t_3)=\\=
  (2\,\mathrm{arccosh}{(3+2 \sqrt{2})},\frac{1}{2};2\,\mathrm{arccosh}{(1+\sqrt{2})},0;2\, \mathrm{arccosh}{(1+\sqrt{2})},0).
\end{gather*}
Another more symmetric decomposition of the Bolza surface into pairs of pants \footnote{derived by Lucy McCarthy in a  project} is one with Fenchel Nielsen paramaters given by
\begin{gather*}
(\ell_1,t_1;\ell_2,t_2;\ell_3,t_3)=(\ell_s,t;\ell_s,t;\ell_s,t), \\
 \ell_s=2\,\mathrm{arccosh}{(1+ \sqrt{2})},\\
 t=\frac{\mathrm{arccosh}\left(\sqrt{\frac{2}{7} \left(3+\sqrt{2}\right)}\right)}{\mathrm{arccosh}
   \left(1+\sqrt{2}\right)}.
\end{gather*}

Note that the Bolza surface is also extremal in the sense that it is the unique maximizer for the length of the systole.

The method of particular solutions can now be applied to the Bolza surface as well. The general code for genus $2$ surfaces was written by Ville Uski (see \cite{strohmaier2013algorithm}).
Based on our paper,

\begin{figure}[htp] 
\centering
\includegraphics*[width=10cm]{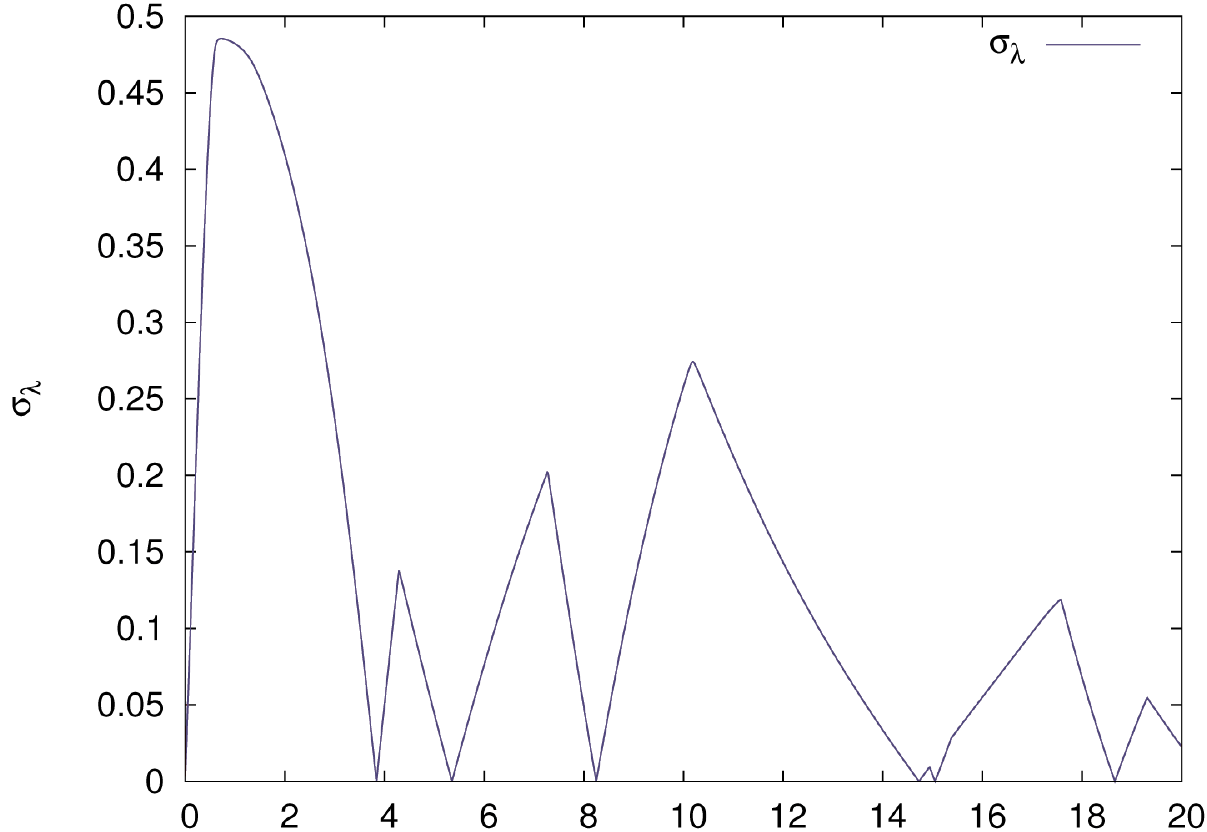}
\caption{Smallest singular value as a function of $\lambda$}\label{bolza-plot}
\end{figure} 

\begin{figure}[htp] 
\centering
\includegraphics*[width=10cm]{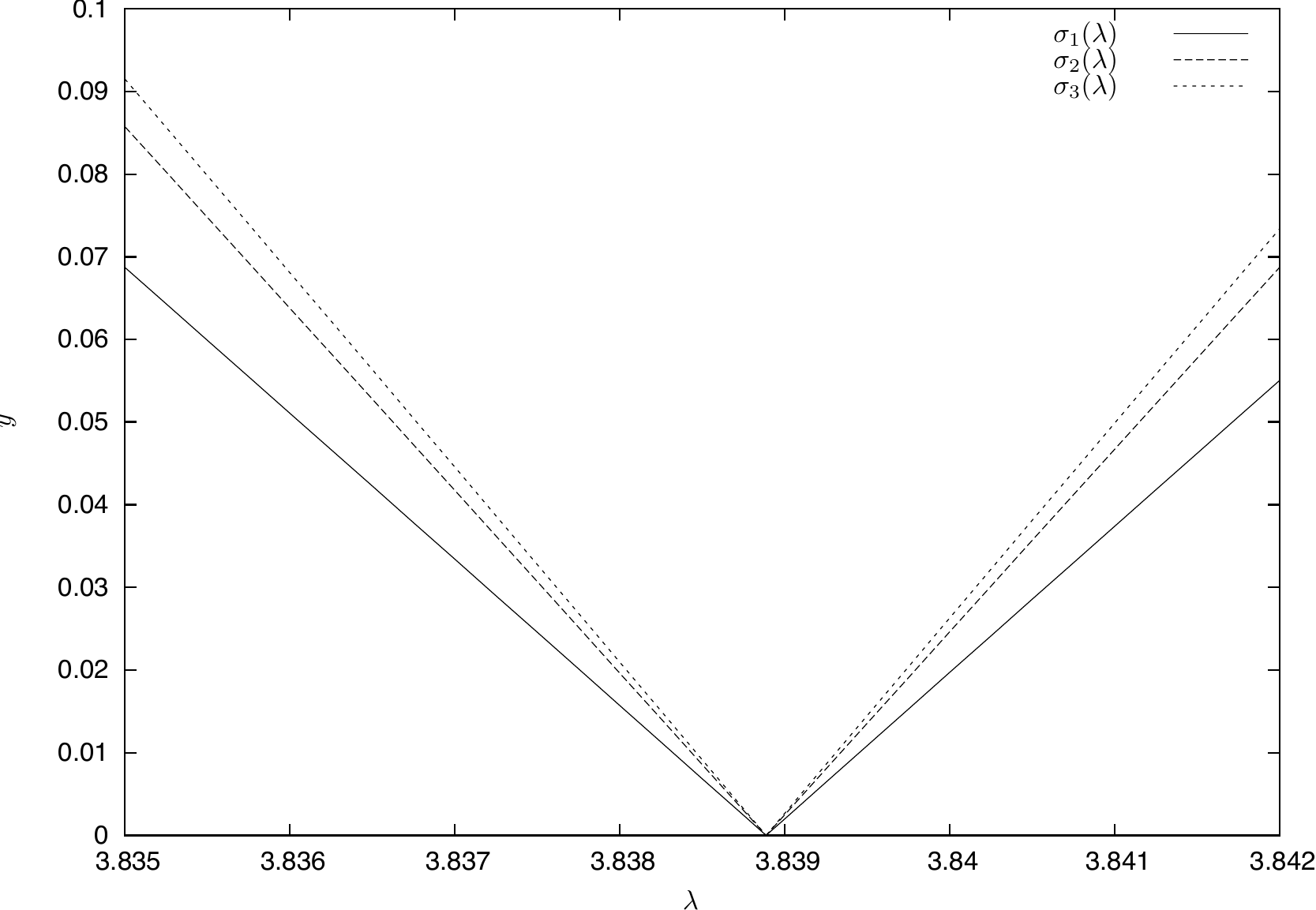}
\caption{Smallest three singular value as a function of $\lambda$}\label{bolza-mag}
\end{figure} 

with high precision, one finds a multiplicity three eigenvalue at
$$
 \lambda_1 = 3.8388872588421995185866224504354645970819150157.
$$
The programme as well as further computed eigenvalues can be found at \url{http://www-staff.lboro.ac.uk/~maas3/publications/eigdata/datafile.html}.
Numerical evidence suggests that this is the global maximum for constant negative curvature genus $2$ surfaces. The reason for it being locally maximal is however its degeneracy.
For an analytic one parameter family of perturbations in Teichm\"uller space one can choose the eigenvalues $\lambda_1,\lambda_2$ and $\lambda_3$ to depend analytically on the perturbation
parameter. Numerically one can see that no matter what perturbation one chooses, none of the eigenvalues  $\lambda_1,\lambda_2$ and $\lambda_3$ has an extremal value at the Bolza surface.
The Bolza surface is also the unique global maximum of the length of the systole. This was shown by Schmutz-Schaller in \cite{schmutz1993reimann}, where more properties of the Bolza surface are discussed.

The following is a list of the first 38 non-zero eigenvalues computed using the method of particular solutions in the implementation described in the paper by Uski and the author in \cite{strohmaier2013algorithm}.\\

\begin{tabular}{c|c}
$\lambda_n$	 & multiplicity\\\hline
        3.83888725884219951858662245043546	&	3\\
         5.35360134118905041091804831103144	&	4\\
	8.24955481520065812189010645068245	&	2\\
	14.7262167877888320412893184421848	&	4\\
	15.0489161332670487461815843402588	&	3\\
	18.6588196272601938062962346613409	&	3\\
	20.5198597341420020011497712606420	&	4\\
	23.0785584813816351550752062995745	&	1\\
	28.0796057376777290815622079450011	&	3\\
	30.8330427379325496742439575604701	&	4\\
	32.6736496160788080248358817081014	&	1\\
	36.2383916821530902525410974752583	&	2\\
	38.9618157624049544290078974084124	&	4\\
\end{tabular}

\section{Heat Kernels, Spectral Asymptotics, and Zeta functions}

Let us start again with general statements. Let $M$ be a $n$-dimensional closed Riemannian manifold and let $-\Delta$ be the Laplace operator
acting on functions on $M$. Assume that $M$ is connected. Then the zero eigenspace is one-dimensional and we can arrange the eigenvalues
such that
$$
 0 = \lambda_0 < \lambda_1 \leq \lambda_2 \leq \ldots
$$
The fundamental solution $k_t(x,y)$ of the heat equation, i.e. the integral kernel of the operator
$e^{t\Delta}$ is well known to be a smoothing operator for all $t>0$. It is hence of trace class and, by Mercer's theorem, we have
\begin{gather}
 \mathrm{tr}(e^{t \Delta})=\sum_{j=0}^\infty e^{-t \lambda_j} = \int_M k_t(x,x) dx.
\end{gather}
For large $t$ one obtains
\begin{gather}
 \mathrm{tr}(e^{t \Delta}) -1 = O(e^{-c t}),
\end{gather}
for some $c>0$.
From the construction of a short time parametrix for the heat equation (see  for example \cite{chavel1984eigenvalues} ) one obtains that as $t\to 0^+$:
\begin{gather}
  \mathrm{tr}(e^{t \Delta}) = t^{-\frac{n}{2}} \sum_{j=0}^N a_j\; t^j + O(t^{N-n/2+1}),
\end{gather}
for any natural number $N$.
The coefficients $a_j$ are integrals of functions $a_j(x)$ that are locally computable from the metric, i.e.
\begin{gather}
 a_j = \int_M a_j(x) dx.
\end{gather}
The first couple of terms are well known
\begin{gather*}
 a_0(x) = \frac{1}{(4 \pi)^{n/2}},\\
 a_1(x) = \frac{1}{(4 \pi)^{n/2}} r(x)/6,
\end{gather*}
where $r$ is the scalar curvature. In two dimensions, the scalar curvature is twice the Gauss curvature so that we have
$a_1(x)=-\frac{1}{12 \pi}$ in the case of a hyperbolic surface, and by Gauss-Bonnet $a_1 = \frac{\mathrm{g}-1}{3}$.

An application of Ikehara's Tauberian theorem to the heat expansion yields Weyl's law that the counting function
$$
 N(\lambda) =\#\{\lambda_j \leq \lambda\}
$$
satisfies
$$
  N(\lambda) \sim C_n \mathrm{Vol}(M) \lambda^{n/2} ,
$$
where $C_n$ depends only on $n$. 

\subsection{Zeta functions}

Because of Weyl's asymptotic formula, the following zeta function is well defined and holomorphic in $s$ for $\Re(s)>\frac{n}{2}$:
$$
 \zeta_\Delta(s) := \sum_{j=1}^\infty \lambda_j^{-s}.
$$
This can easily be rewritten as
$$
  \zeta_\Delta(s) = \frac{1}{\Gamma(s)} \int_0^\infty t^{s-1} \left( \mathrm{tr} (e^{t \Delta}) -1 \right) dt.
$$
We can now split this integral into two parts to obtain
$$
 \Gamma(s)  \zeta_\Delta(s) = \int_0^1 t^{s-1} \left( \mathrm{tr} (e^{t \Delta}) -1 \right) dt +\int_1^\infty t^{s-1} \left( \mathrm{tr} (e^{t \Delta}) -1 \right) dt = I_1(s) + I_2(s).
$$
Note that $I_2(s)$ is entire in $s$. The integral $I_1(s)$ can be rewritten using the asymptotic expansion
$$
 I_1(s) = \int_0^1 t^{s-1} \left( \mathrm{tr} (e^{t \Delta}) -  t^{-\frac{n}{2}} \sum_{j=0}^N  a_j\; t^j  \right) dt +  \sum_{j=0}^N \int_0^1 a_j \; t^{j+s-1-\frac{n}{2}} dt -  
 \int_0^1 t^{s-1} dt.
$$
The last two terms together yield
$$
  -\frac{1}{s}+\sum_{j=0}^N \frac{a_j}{s+j-\frac{n}{2}},
$$
and the first integral is holomorphic for $\Re{s}>\frac{n}{2}-N$. This can be done for any natural number $N$.
Therefore, $I_1(s)$ has a meromorphic extension to the entire complex plane with simple poles at
$\frac{n}{2}-j$ and at $0$. Hence, we showed that $\zeta$ admits a meromorphic extension to the complex plane.
Since $\Gamma(s)$ has a pole at the non-positive integers this shows that $\zeta$ is regular at all the non-positive integers. In particular zero is not a pole of $\zeta$.
The above shows that $\zeta_\Delta(0)=-1$ if $n$ is odd and $\zeta_\Delta(0)= -1 + a_{\frac{n}{2}}$ if $n$ is even.
The value $\zeta_\Delta'(0)$ is therefore well defined and is used to define the zeta-regularized determinant $\det_\zeta(-\Delta)$ of $-\Delta$ as follows
$$
  \zeta'_\Delta(0)= - \log \mathrm{det}_\zeta(-\Delta).
$$
The motivation for this definition is the formula
$$
\log \det (A) = \sum_{j=1}^N \log \lambda_j = \left( - \frac{d}{ds} \sum_{j=1}^N \lambda_j^{-s} \right)|_{s=0},
$$
for a non-singular Hermitian $N \times N$-matrix with eigenvalues $\lambda_1,\ldots,\lambda_N$.

The computation of this spectral determinant is quite a challenge. The method of meromorphic continuation for the zeta function also is a method of computation
for the spectral determinant.

\section{The Selberg Trace Formula}

Suppose that $M$ is an connected oriented hyperbolic surface. Then there is an intriguing formula connecting the spectrum of the Laplace operator
to the length spectrum. 
Suppose that $g \in C^\infty_0(\R)$ is an even real valued test function. Then its Fourier transform $h = \hat g$ is an entire function defined on the entire complex plane. It is also in the Schwartz space $\mathcal{S}(R)$
and real valued on the real axis. As usual, we use the notation  $\lambda_j = r_j^2 +\frac{1}{4}$, where for eigenvalues smaller than $\frac{1}{4}$ we choose $r_j$ to have positive
imaginary part.
Hence, by Weyl's law, the sum 
$$
 \sum_{\lambda_j} h\left (\sqrt{\lambda_j - \frac{1}{4}} \right) =  \sum_{j} h(r_j)
$$
converges and depends continuously on $g$. It therefore defines an even distribution
$$
  \mathrm{Tr}\cos \left ( t \sqrt{\Delta - \frac{1}{4}} \right) 
$$
 in $\mathcal{D}'(\R)$.
Selberg's trace formula reads

 \begin{eqnarray*}
  \sum_{n=0}^{\infty} h(r_n) & = & \frac{\mathrm{Vol}(M)}{4 \pi}
  \int_{-\infty}^{\infty} r h(r) \tanh (\pi r) dr \nonumber + 
\sum_{k =1}^{\infty} \sum_{\gamma}\frac{\ell(\gamma)}{2\sinh(k \ell(\gamma)/2)}g(k \ell(\gamma)),
\end{eqnarray*}
where the second sum in the second term is over the set of primitive closed geodesics $\gamma$, whose length is denoted by $\ell(\gamma)$.
We would like to refer to Iwaniec's monograph \cite{iwaniec2002spectral} for an introduction and a derivation.
In the sense of distributions this reads as follows.
\begin{eqnarray*}\label{eqn:ofldwave}
 \mathrm{Tr}\cos \left( t \sqrt{\Delta - \frac{1}{4}} \right) & = &
 -\frac{\mathrm{Vol}(M)}{8 \pi} \frac{\cosh (t/2)}{\sinh^2(t/2)} +  \sum_{k =1}^{\infty} \sum_{\gamma}\frac{\ell(\gamma)}{4\sinh(k \ell(\gamma)/2)}(\delta(|t|-k \ell(\gamma))).
\end{eqnarray*}
Note that this is not a tempered distribution. Therefore, we may not pair either side with a general Schwartz functions. One can however still apply it to the function 
$h(x) = e^{-t x^2}$ and obtain
\begin{gather*} \label{stf}
 \mathrm{tr}(\mathrm{e}^{\Delta t}) = \frac{\mathrm{Vol}(M)\mathrm{e}^{-\frac{t}{4}}}{4 \pi t} \int_0^\infty \frac{\pi \mathrm{e}^{-r^2 t}}{\cosh^2(\pi r)} dr +
 \sum_{n=1}^\infty \sum_{\gamma}  \frac{\mathrm{e}^{-t/4}}{\sqrt{4 \pi t}} 
  \frac{\ell(\gamma) \mathrm{e}^{-\frac{n^2 \ell(\gamma)^2}{4 t}}}{2 \sinh\frac{n \ell(\gamma)}{2}}.
\end{gather*}
Note that the second term is of order $O(e^{-\frac{\ell_0^2}{4t}})$ as $t \to 0^+$, where $\ell_0$ is the length of the shortest closed geodesic (the systole length).
The first term can therefore be thought of as a much more refined version of the heat asymptotics.
\begin{exercise} 
 Derive the heat asymptotics from the first term in Selberg's trace formula by asymptotic analysis. Derive the first three heat coefficients.
\end{exercise}

The formula
$$
  \zeta_{\Delta}(s) = \frac{1}{\Gamma(s)} \int_0^\infty t^{s-1} \left(  \mathrm{tr}(\mathrm{e}^{\Delta t}) -1 \right) dt
$$
can now directly be used with the Selberg trace formula. In order to perform the analytic continuation, one can again split the integral
into integrals over $(0,1]$ and over $(1,\infty)$. For numerical purposes it is however convenient to instead split into $(0,\epsilon]$ and $(\epsilon,\infty)$ for a suitably chosen $\epsilon>0$.
This means
$$
   \zeta_{\Delta}(s) = \frac{1}{\Gamma(s)} \int_0^\epsilon t^{s-1} \left(  \mathrm{tr}(\mathrm{e}^{\Delta t}) -1 \right) dt + \frac{1}{\Gamma(s)} \int_\epsilon^\infty t^{s-1} \left(  \mathrm{tr}(\mathrm{e}^{\Delta t}) -1 \right) dt.
$$
We now compute the first term from the Selberg trace formula and the second term from the spectrum. Using the same unique continuation process as described earlier,
one obtains the following representation of the spectral zeta function for $\Re(s)> -N$:
$$
  \zeta_{\Delta}(s) =  \frac{1}{\mathrm{\Gamma}(s)}( T_1^{\epsilon}(s)+ T_2^{\epsilon,N}(s) + 
  T_3^{\epsilon,N}(s) + T_4^{\epsilon,N}(s)),
$$
where
\begin{gather*}
  T_1^\epsilon(s) =  \sum_{i=1}^\infty \lambda_i^{-s} \mathrm{\Gamma}(s,\epsilon \lambda_i),\\
  T_2^{\epsilon,N}(s) = \sum_{k=0}^N \frac{a_k \epsilon^{s+k-1}}{s+k-1},\\
  T_3^{\epsilon,N}(s) = \frac{\mathrm{Vol}(M)}{4 \pi} \int_0^\infty  I_N^\epsilon(r) dr,\\ 
  T_4^{\epsilon,N}(s) = \sum_{n=1}^\infty \sum_{\gamma} \int_0^\epsilon t^{s-1} \frac{\mathrm{e}^{-t/4}}{\sqrt{4 \pi t}} 
  \frac{\ell(\gamma) \mathrm{e}^{-\frac{n^2 \ell(\gamma)^2}{4 t}}}{2 \sinh\frac{n \ell(\gamma)}{2}} dt.
\end{gather*}
Here
$$
  I_N^\epsilon(r) = \int_0^\epsilon t^{s-2}  \left( \mathrm{e}^{-(r^2+\frac{1}{4}) t } - 
 \sum_{k=0}^N \frac{(-1)^k}{k!} (r^2+\frac{1}{4})^k t^k \right) dt,
$$
and the coefficients $a_k$ are the heat coefficients of the expansion of  $\mathrm{tr}(\mathrm{e}^{\Delta t}) -1$, which are given by
$$
 a_k = \frac{\mathrm{Vol}(M)}{4 \pi} \int_0^\infty \frac{(-1)^k}{k!} \frac{\pi (r^2+1/4)^k}{\cosh^2(\pi r)} dr - \delta_{1,k}.
$$
As usual $\mathrm{\Gamma}(x,y)$ denotes the incomplete Gamma function
$$
 \mathrm{\Gamma}(x,y) = \int_y^\infty t^{x-1} \mathrm{e}^{-t} dt.
$$
Differentiation gives the following formula for the spectral determinant.
$$
 -\log {\det}_{\zeta} \Delta = \zeta'_\Delta(0) = L_1^\epsilon + L _2^\epsilon + L_3^\epsilon
$$
where
\begin{gather*}
 L_1^\epsilon = \sum_{i=1}^\infty \mathrm{\Gamma}(0,\epsilon \lambda_i),\\
 L_2^\epsilon = - \frac{\mathrm{Vol}(M)}{4 \pi \epsilon} -\left(\frac{\mathrm{Vol}(M)}{12 \pi} + 1\right)(\gamma + \log(\epsilon)) 
  +  \frac{\mathrm{Vol}(M)}{4} \times \\ \int_0^\infty 
 \text{sech}^2(\pi  r) \left( \frac{1-\mathrm{E}_2 \left(\epsilon (r^2+\frac{1}{4})\right)}{\epsilon}+(r^2+\frac{1}{4})
  \left(\gamma-1+\log(\epsilon(r^2+1/4)) \right) \right) dr,\\
  L_3^\epsilon =  \sum_{n=1}^\infty \sum_{\gamma} \int_0^\epsilon \mathrm{e}^{-t/4} \frac{\ell_i \mathrm{e}^{-\frac{n^2 \ell(\gamma)^2}{4 t}}}{4 \sqrt{\pi } t^{3/2} \sinh{\left(\frac{1}{2} n \ell(\gamma) \right)}} dt,
\end{gather*}
and $\mathrm{E}_2(x)$ is the generalized exponential integral which equals $x \;\mathrm{\Gamma}(-1,x)$.
All the integrals have analytic integrands and can be truncated with exponentially small error. They can therefore
be evaluated to high accuracy using numerical integration.

For fixed $s$ and $\epsilon>0$ not too small the sums over the eigenvalues converge very quickly and therefore 
$T_1^\epsilon(s)$ and $L_1^\epsilon$ can be computed accurately from the first eigenvalues only.

If $\epsilon$ is small compared to $\ell_0^2$ the terms $T_4^{\epsilon,N}(s)$ and $L_3^\epsilon$ are very small. The terms $L_1^\epsilon$ and $T_1^\epsilon(s)$ involve the spectrum
but the sums converge rapidly, so that only a finite proportion of the spectrum is needed to numerically approximate these values. A detailed error analysis of these terms is carried out in \cite{mroz2014explicit}.

In order to illustrate the idea behind this method, let us look at the function
$$
 R_N(t)=\sum_{j=0}^N e^{-\lambda_j t} - \frac{\mathrm{Vol}(M)\mathrm{e}^{-\frac{t}{4}}}{4 \pi t} \int_0^\infty \frac{\pi \mathrm{e}^{-r^2 t}}{\cosh^2(\pi r)} dr.
$$ 
By Selberg's trace formula we have
$$
  R_N(t)=-\sum_{j=N+1}^\infty e^{-\lambda_j t} +  \sum_{n=1}^\infty \sum_{\gamma}  \frac{\mathrm{e}^{-t/4}}{\sqrt{4 \pi t}} 
  \frac{\ell(\gamma) \mathrm{e}^{-\frac{n^2 \ell(\gamma)^2}{4 t}}}{2 \sinh\frac{n \ell(\gamma)}{2}}.
$$
The first term is negative and dominant when $t$ is small. The second term is positive and dominates when $t$ is large.
Figure \ref{heat-comparison} shows this function for the Bolza surface. Here the first $500$ eigenvalues were computed numerically using the method outlined in the previous paragraphs.
The integral in the zero term of the Selberg trace formula is computed numerically.
\begin{figure}[htp] 
\centering
\includegraphics*[width=10cm]{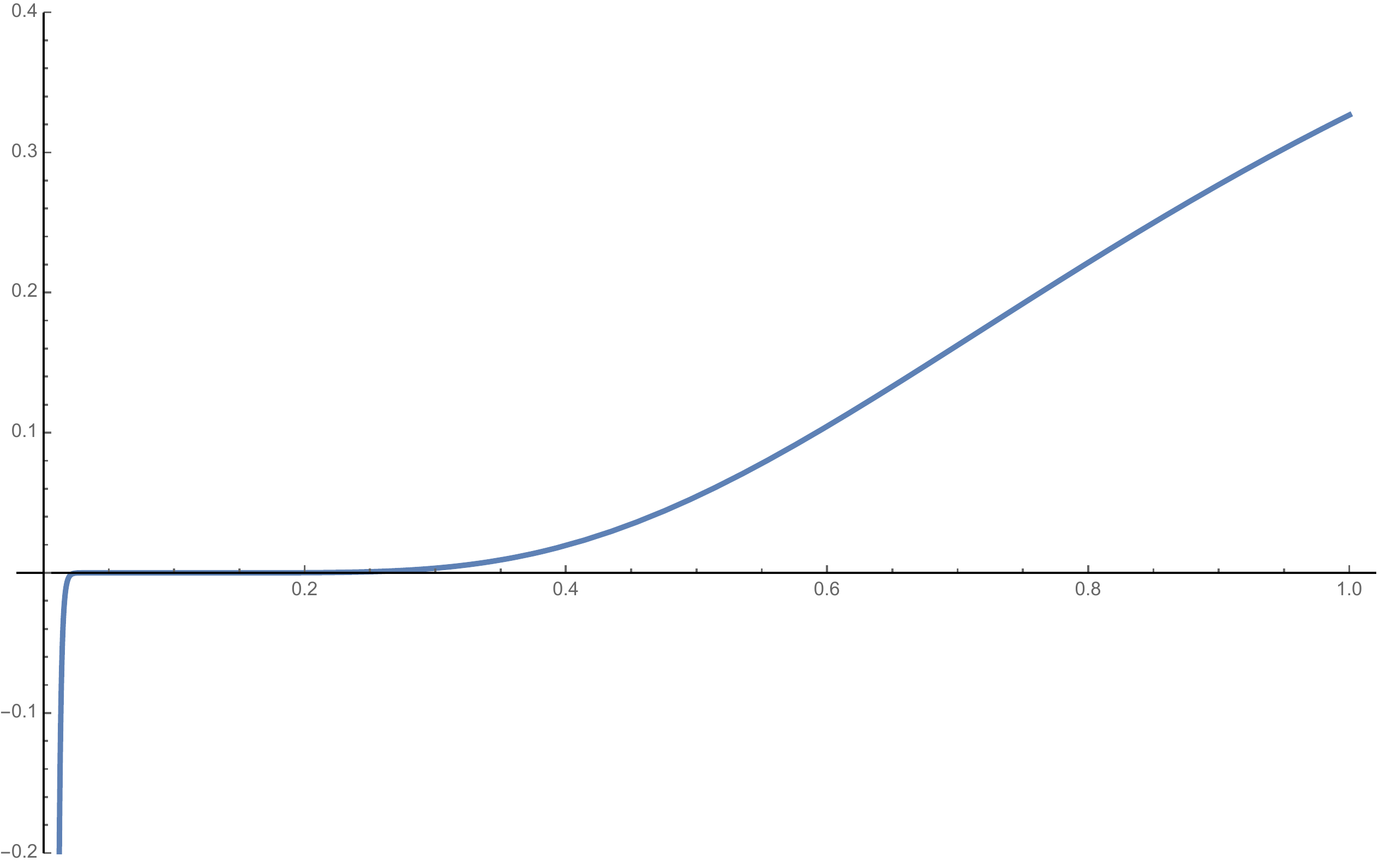}
\caption{The function $R_N$ for the Bolza surface with $N=500$}\label{heat-comparison}
\end{figure} 

One can now clearly see the regions in which each term dominates. There is a clearly visible region between $t=0.05$ and $t=0.2$ where the function is very small.
In fact its value at $t=0.1$ is of order smaller than $10^{-9}$.

In order to compute the spectral zeta function one can therefore choose $\epsilon=0.1$ and estimate the errors of the contributions of $T_4$ and $L_3$,
as well as the error from cutting off the spectrum and considering only the first $500$ eigenvalues.
One obtains for example for the Bolza surface
\begin{gather*}
 \mathrm{det}_\zeta(\Delta) \approx 4.72273,\\
 \zeta_\Delta(-1/2) \approx -0.650006.
\end{gather*}
To compute the first $500$ eigenvalues of the Bolza surface to a precision of $12$ digits, about $10000$
$\lambda$-evaluations of generalized singular value decomposition were needed. This took about
$10$ minutes on a $2.5$ GHz Intel Core i5 quad core processor (where parallelization was used).

\begin{figure}[htp]
\centering \label{zetaplot}
\includegraphics*[width=11cm]{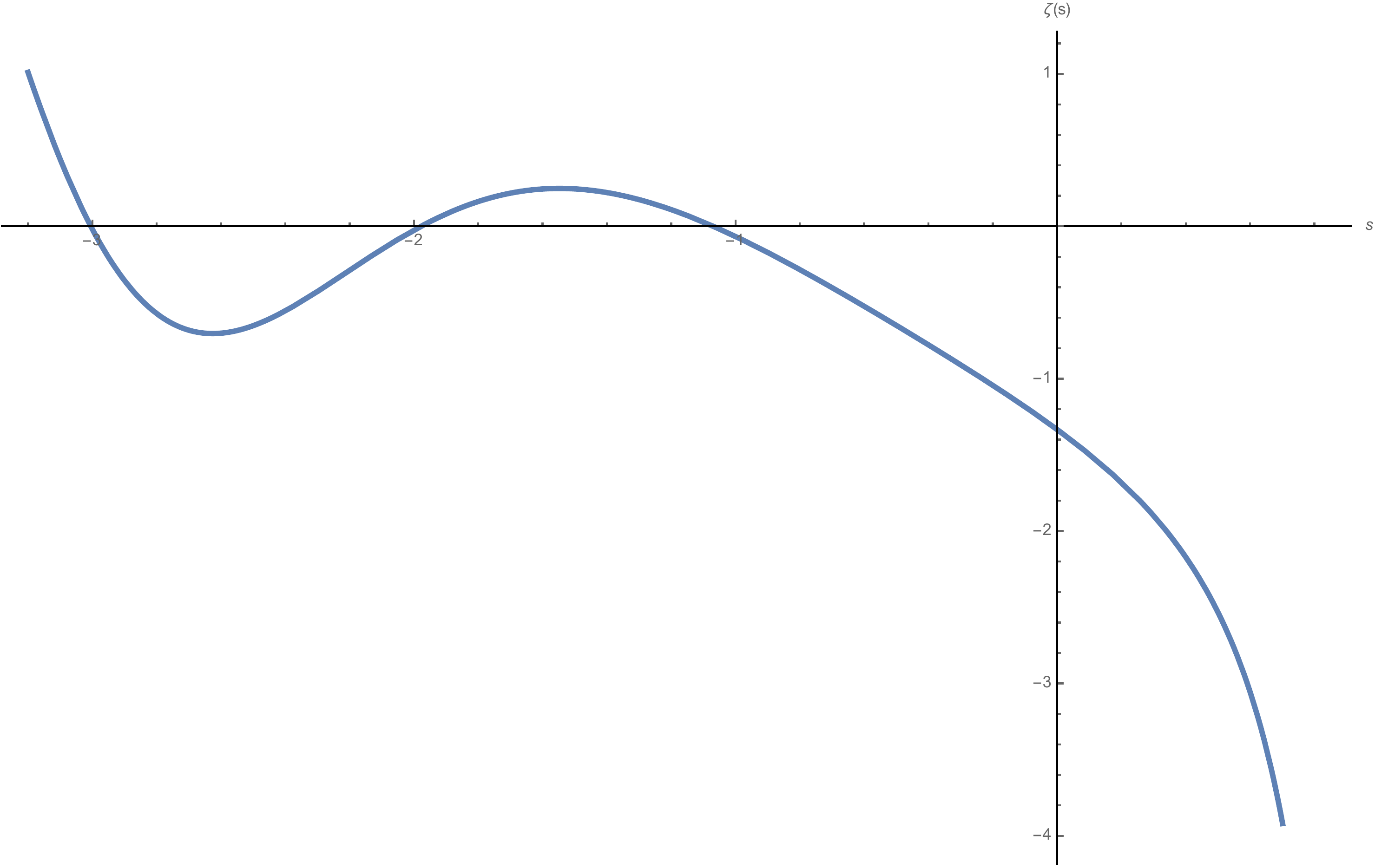}
\caption{$\zeta_\Delta(s)$ as a function of $s$ for the Bolza surface}
\end{figure} 

Numerical evidence suggests that the spectral determinant is maximized in genus $2$ for the Bolza surface. One can see quite clearly from perturbing in Teichm\"uller space that the Bolza surface 
is indeed a local maximum for the spectral determinant. Note that the Bolza surface is known to be a critical point by symmetry considerations.

\section{Completeness of a Set of Eigenvalues}

The method of particular solution on oriented hyperbolic surfaces is able to produce quite quickly a list of eigenvalues. Once such a list is computed and error bounds are established, one would like to check that this list is complete and one has not missed an eigenvalue, for example because the step-size in the search algorithm was chosen too small, or an eigenvalue had a higher multiplicity. In \cite{strohmaier2013algorithm} it was proved that
the step size can always be chosen small enough so that no eigenvalues are missed. Choosing the step-size according to these bounds does however slow down the speed of computation significantly.
In this section we discuss two methods by which completeness of a set of eigenvalues can be checked.

\subsection{Using the heat kernel and Selberg's trace formula}

Suppose that $\{ \mu_0, \ldots, \mu_N \}$ is a list of computed eigenvalues. We would like to use this list and check that there are no additional eigenvalues in an interval $[0,\lambda]$,
where $\lambda$ is possibly smaller than $\mu_N$. 
As before consider the function
$$
 R_N(t)=\sum_{j=0}^N e^{-\lambda_j t} - \frac{\mathrm{Vol}(M)\mathrm{e}^{-\frac{t}{4}}}{4 \pi t} \int_0^\infty \frac{\pi \mathrm{e}^{-r^2 t}}{\cosh^2(\pi r)} dr,
$$ 
and recall that
$$
  R_N(t)= -\sum_{j=N+1}^\infty e^{-\lambda_j t} +  \sum_{n=1}^\infty \sum_{\gamma}  \frac{\mathrm{e}^{-t/4}}{\sqrt{4 \pi t}} 
  \frac{\ell(\gamma) \mathrm{e}^{-\frac{n^2 \ell(\gamma)^2}{4 t}}}{2 \sinh\frac{n \ell(\gamma)}{2}}.
$$
For $t<T<\sqrt{\ell_0^2+1}-1$ the second term is bounded by 
\[
F_T(t)=\sqrt{\frac{T}{t}} \mathrm{tr}(e^{-\Delta T}) e^{\frac{T}{4}+\frac{l^2}{4 T}} e^{\frac{-l^2}{4 t}},
\]
In  \cite{mroz2014explicit} Fourier Tauberian theorems were used to establish the bound
$$
F_T(t) \! \leq \! \frac{\mathrm{Vol}(M)}{4 \pi} \frac{1}{\sqrt{t}} e^{\frac{T}{4}+\frac{\ell_0^2}{4 T}-\frac{\ell_0^2}{4 t}} \! \left( \frac{1}{ \sqrt{T}} +\frac{2 \nu^2+\nu \pi}{ \sqrt{\pi} \ell_0 } +\sqrt{T}\left(\frac{4\nu^3+2 \nu^2 \pi}{ \pi \ell_0^2} \right) \right),
$$
where
$\nu \approx 4.73$ is the first non-zero solution to the equation $\cos(\lambda) \cosh(\lambda)=1$, and $\ell_0$ is the systole length.
Hence,
$$
 R_N(t) \leq  F_T(t).
$$
Therefore, if we compute
$$
 \tilde R_N(t)=\sum_{j=0}^N e^{-\mu_j t} - \frac{\mathrm{Vol}(M)\mathrm{e}^{-\frac{t}{4}}}{4 \pi t} \int_0^\infty \frac{\pi \mathrm{e}^{-r^2 t}}{\cosh^2(\pi r)} dr
$$ 
and
$$
 F_T(t) - \tilde R_N \leq \epsilon,
$$
then there can not be any additional eigenvalues in the interval $[0,-\frac{\log{F_T(t) - \tilde R_N }}{t}]$ as otherwise we would have
$$
 R_N(t) > F_T(t).
$$

For the Bolza surface we have $\ell_0 \approx 3.05714$ and we can choose for instance $T=2$. 
\begin{figure}[htp]
\centering 
\includegraphics*[width=13cm]{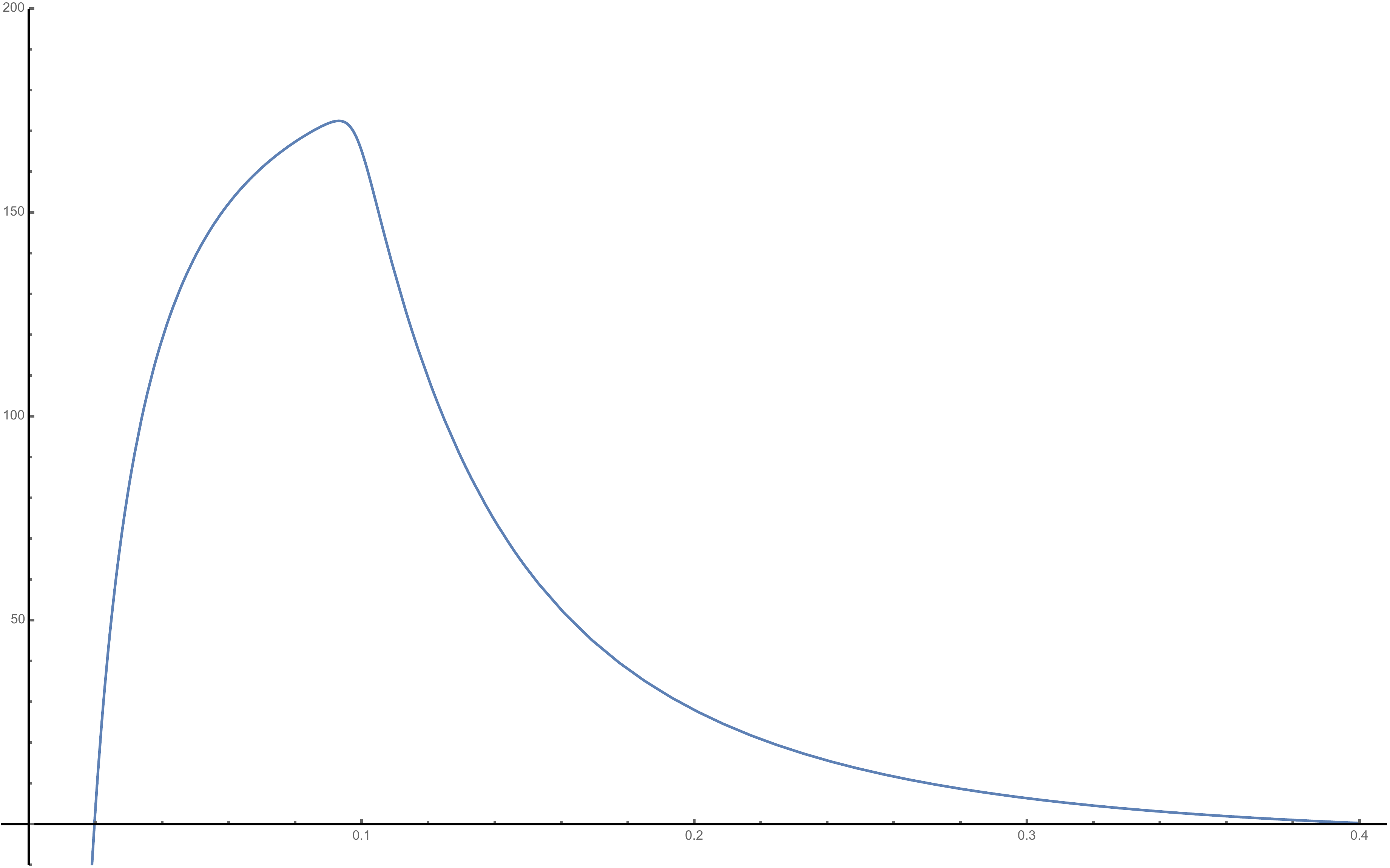}
\caption{$-\frac{\log{F_T(t) - \tilde R_N }}{t}$ as a function of $t$ for the Bolza surface, $N=200$} \label{missedevplot}
\end{figure} 

Using the list of the first $200$ eigenvalues one can see from Fig. \ref{missedevplot} that choosing $t$ near $0.1$ maximizes the function  $-\frac{\log{F_T(t) - \tilde R_N }}{t}$.
For $t=0.095$ one gets that there are no additional eigenvalues smaller than $172$. Note that $\lambda_{200} \approx 200.787$. So we had to compute roughly $30$ more eigenvalues to make sure our list is complete.
This method in principle can be made rigorous by using interval arithmetics. Its disadvantage is that for larger lists it requires the low lying eigenvalues to be known with very high accuracy.

\subsection{Using the Riesz mean of the counting function}

It is sometimes convenient to reparametrize in terms of square roots of eigenvalues. 
Let us define the local counting function $$\tilde N(t) = N(t^2) = \#\{\lambda_j \leq t^2 \} = \#\{\sqrt{\lambda_j} \leq t \}.$$
For a general negatively curved two dimensional compact Riemannian manifold one has (see \cite{berard1977wave})
$$
 \tilde N(t) \sim \frac{\mathrm{Vol}(M)}{4 \pi} t^2 + O(\frac{t}{\log(t)}),
$$
as $t \to \infty$. Because of the growing error term this is unsuitable to detect missed eigenvalues from the spectrum.
However, the so-called Riesz means of the counting functions are known to have improved asymptotic expansions. 
In our case define the first Riesz mean as 
$$
 (R_1  \tilde N) (t) := \frac{1}{t} \int_0^t \tilde N(r) dr. 
$$
Then for two dimensional compact surfaces of negative curvature one has
$$
 (R_1  \tilde N) (t) = \frac{\mathrm{Vol}(M)}{12 \pi} t^2 + \frac{1}{12 \pi} \int \kappa(x) dx + O(\frac{1}{\log(t)^2}),
$$
where $\kappa(x)$ is the scalar curvature at the point $x \in M$. This can be inferred in the case of constant curvature hyperbolic surfaces from Selberg's trace formula (see \cite{hejhal1976selberg}),
but also can be shown to hold true  in the case of negative variable curvature (\cite{mroz2014riesz}).
In the case of hyperbolic surfaces one obtains
$$
 (R_1  \tilde N) (t) = \frac{\mathrm{Vol}(M)}{12 \pi} \left(t^2 - 1 \right) + O(\frac{1}{\log(t)^2}).
$$

The strategy is to compute the Riesz means from a set of computed eigenvalues. That is, if $\{\mu_0,\ldots, \mu_N\}$ is a set of eigenvalues
we compute the function
$$
 \tilde N_{test}(t) := \#\{\sqrt{\mu_j} \leq t \}
$$
and plot
$$
 F_{test(t)}:=(R_1 \tilde N_{test})(t) -  \frac{\mathrm{Vol}(M)}{12 \pi} \left(  t^2 - 1 \right).
$$
This is done in Fig. \ref{missed-ev-plot2} for the Bolza surface. The red line was computed with an eigenvalue missing. One can clearly see this in the plot, and this also allows one to say roughly where the missing eigenvalue was.
If an eigenvalue is missing somewhere this will result in the function not going to zero. In this way one can even detect roughly where the missed eigenvalue is located and how many eigenvalues may be missing.

\begin{figure}[htp]
\centering
\includegraphics*[width=13cm]{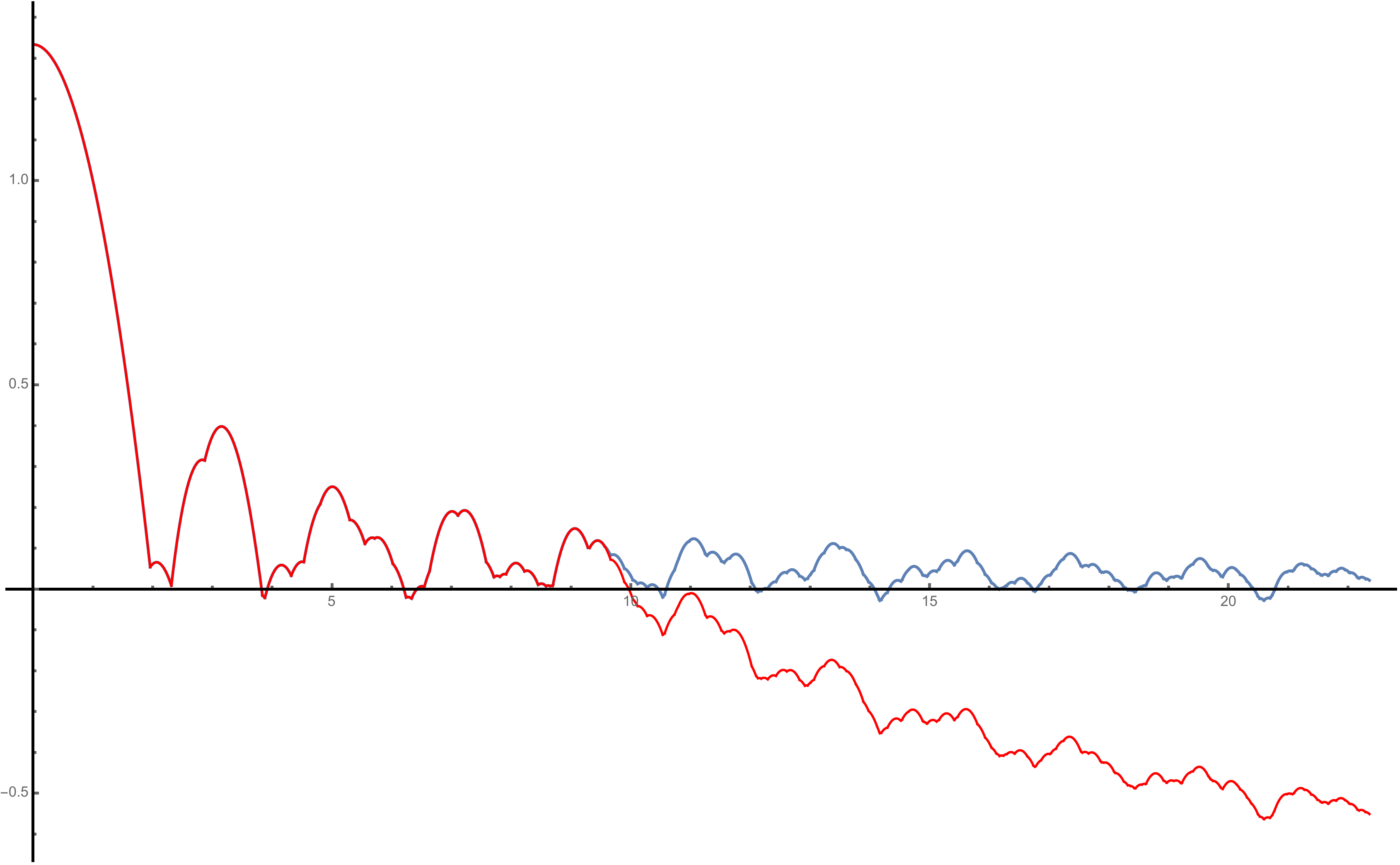}
\caption{$F_{test(t)}$ as a function of $t$ for the Bolza surface, the red line is the function with $\lambda_{89} \approx (9.563)^2$ missing} \label{missed-ev-plot2}
\end{figure} 

\noindent{\bf Acknowledgements.} 

I would like to thank the organizers  of the summer school for the perfect organization and the hospitality.
I am also grateful to Joseph Cook for carefully reading these notes and for providing some numerical work on the Bolza surface as well as diagrams.

\newpage


\end{document}